\documentclass[11pt]{amsart}
\usepackage{rotating}
\usepackage{multirow}

\usepackage[algoruled,vlined,norelsize]{algorithm2e}

\usepackage[T1]{fontenc}
\usepackage[utf8]{inputenc}
\usepackage[english]{babel}
\usepackage{latexsym,amsfonts,amsmath,amssymb,amsthm,amscd,stmaryrd,float, hyperref,bm, fullpage}

\usepackage{graphicx,tikz}
\usetikzlibrary{patterns}

\usepackage{hyperref}%

\usepackage{bm, makecell, todonotes, caption}%

\newcounter{Change}%

\usepackage[refpage,intoc]{nomencl}%

\newcounter{nomcount}%
\makeatletter%
\newcommand{\nomentry}[2]{\stepcounter{nomcount}\nomenclature[\two@digits\thenomcount]{#1}{#2}}%
\makeatother%
 \makenomenclature%

\makeatletter
\renewcommand\section{\@startsection
	{section}% nom du titre
{1}% niveau de titre
{0pt}% indentation
{-3.5ex plus -1ex minus -.2ex}%espace vertical avant 
{2.3ex plus.2ex}% espace vertical après 
{\centering\normalfont\Large\scshape}}

\renewcommand\subsection{\@startsection
	{subsection}% nom du titre
{2}% niveau de titre
{0pt}% indentation
{-3ex plus -1ex minus -.2ex}%espace vertical avant 
{1ex plus.2ex}% espace vertical après 
{\normalfont\large\bfseries}}

\renewcommand\subsubsection{\@startsection
	{subsubsection}% nom du titre
{3}% niveau de titre
{0pt}% indentation
{-1.5ex plus -1ex minus -.2ex}%espace vertical avant 
{0.8ex plus .2ex}% espace vertical après 
{\normalfont\bfseries}}

\renewcommand\paragraph{\@startsection
	{paragraph}% nom du titre
{4}% niveau de titre
{0em}% indentation
{-1.2ex plus -0.4ex minus -.2ex}%espace vertical avant 
{0ex}% espace vertical après 
{\normalfont\bfseries}}

\makeatother

\newcommand\A{\mathcal A}%
\newcommand\B{\mathcal B}%
\newcommand\Z{\mathbb Z}%
\newcommand\N{\mathbb N}%
\newcommand\R{\mathbb R}%
\newcommand\Nb{\mathcal N}%
\newcommand\M{\mathcal M}%
\newcommand\Ball{\mathbf B}%
\newcommand\F{F_\ast}%
\newcommand\V{\mathcal V}%
\renewcommand\P{\mathcal P}%
\renewcommand\L{\mathcal L}%
\newcommand\s\sigma%
\newcommand\az{\A^\Z}%
\newcommand\bz{\B^\Z}%
\newcommand\dm{d_\M}%
\newcommand\ud{\mathrm d}%
\newcommand\sk\bigskip%
\newcommand\per[1]{\mathstrut^\infty#1^\infty}%
\newcommand\tx{\text}%
\newcommand\define\textit%
\newcommand\Ba{\mathfrak{B}}%
\newcommand\Ms{\M_{\sigma}}%
\newcommand\Merg{\mathcal{M}_{\s-\textrm{erg}}}%
\newcommand\Msaz{\Ms(\az)}%
\newcommand\meas[1]{\widehat{\delta_{#1}}}%
\newcommand\Ber{\mathrm {Ber}}%
\newcommand\card{\textrm{Card}}%
\newcommand\freq{\textrm{Freq}}%
\newcommand\supp{\mathrm{supp}}%
\newcommand\argmin[1]{\underset{#1}{\mbox{argmin }}}%
\newcommand\p{\textrm{Part}}%
\newcommand\prog{\textrm{Prog}}%
\newcommand\inter{\textrm{Inter}}%

\newcommand{\band}{\mbox{ and }}
\newcommand{\bor}{\mbox{ or }}
\newcommand\D{\mathcal D}
\newcommand\aarz{(\A^{\A^\Nb})^\Z}

\theoremstyle{definition}
\newtheorem{definition}{Definition}

\theoremstyle{plain}
\newtheorem{lemma}{Lemma}
\newtheorem{theorem}{Theorem}

\newtheorem{proposition}{Proposition}

\newtheorem{corollary}{Corollary}

\newcounter{claimcount}[theorem]

\newcommand{\bclaimprf}[1][Proof.~]{\begin{list}{}{\setlength{\leftmargin}{0.5em}
\setlength{\rightmargin}{2em}\setlength{\listparindent}{1em}}\item
{\em #1}}

\newcommand{\eclaimprf}{ \hfill $\Diamond$~{\scriptsize {\tt Claim~\theclaimcount}}\end{list}} % %

\usepackage{xcolor}
\usepackage{framed}

\title[Self-organisation in cellular automata with coalescent particles: Qualitative and quantitative approaches]{Self-organisation in cellular automata with coalescent particles: Qualitative and quantitative approaches}

\author{Benjamin Hellouin de Menibus}
\address{Andrés Bello University and Centro de Modelamiento Matemático (Universidad de Chile), Santiago. }
\email{benjamin.hellouin@gmail.com}
\urladdr{http://mat-unab.cl/~hellouin/}

\author{Mathieu Sablik}
\address{Aix Marseille Université, CNRS, Centrale Marseille, I2M UMR 7373\\
13453, Marseille, France}
\email{sablik@univ-amu.fr}
\urladdr{http://www.i2m.univ-amu.fr/~sablik/}
\keywords{Cellular automata, Particles, limit measures, Brownian motion}

\date{}
\begin{document}

\begin{abstract}
This article introduces new tools to study self-organisation in a family of simple cellular automata which contain some particle-like objects with good collision properties (coalescence) in their time evolution. We draw an initial configuration at random according to some initial shift-ergodic measure, and use the limit measure to describe the asymptotic behaviour of the automata.  

We first take a qualitative approach, i.e. we obtain information on the limit measure(s). We prove that only particles moving in one particular direction can persist asymptotically. This provides some previously unknown information on the limit measures of various deterministic and probabilistic cellular automata: $3$ and $4$-cyclic cellular automata (introduced by Fisch, 1990), one-sided captive cellular automata (introduced by Theyssier, 2004), the majority-traffic cellular automaton, a self stabilisation process towards a discrete line (introduced by Regnault and Remila, 2015)\dots 

In a second time we restrict our study to to a subclass, the gliders cellular automata. For this class we show quantitative results, consisting in the asymptotic law of some parameters: the entry times (generalising~\cite{EntryTimes}), the density of particles and the rate of convergence to the limit measure.
\end{abstract}

\maketitle

 \section{Introduction} 

A cellular automaton is a complex system defined by a local rule which acts synchronously and uniformly on a configuration space. These simple models exhibit a wide variety of dynamical behaviour and even in the one-dimensional case (the focus of this article) they are not completely understood. Formally, given a finite alphabet $\A$, a configuration is an element of the set $\az$. This set is compact for the product topology. A \define{cellular automaton} $F:\az\to\az$ is defined by a local function $f:\A^{[-r,r]}\to \A$, for some \emph{radius} $r>0$, which acts synchronously and uniformly on every cell of the configuration:
\[F(x)_i=f(x_{[i-r,i+r]})\textrm{ for all }x\in\az\textrm{ and }i\in\Z.\]
Equivalently, cellular automata can be defined as continuous functions that commute with the \emph{shift map} $\sigma$ defined by $\sigma(x)_i=x_{i+1}$ for all $x\in\az$ and $i\in\Z$.

Even though cellular automata have been introduced by J. Von Neumann~\cite{VonNeumann}, the impulsion for their systematic study was given by the work of S. Wolfram~\cite{Wolfram84}. He instigated a systematic study of \emph{elementary} cellular automata, which are the cellular automata defined on the alphabet $\{0,1\}$ with radius $1$ (there are $2^{2^3}=256$ such cellular automata; to each of them we associate a number $\# n$). In particular he proposed a classification according to the observation of the space-time diagrams produced by the time evolution of cellular automata starting from a random configuration. 

One of these classes corresponds to a particular form of self-organisation: from a random configuration, after a short transitional regime, regions consisting in a simple repeated pattern emerge and grow in size, while the boundaries between them persist under the action of the cellular automaton and can be followed from an instant to the next. Therefore their movement (time evolution) can be defined inductively, and in this case we call these boundaries \define{particles}. In the simplest case, these particles evolve at constant speed and are annihilated when colliding with other particles; however, they can sometimes exhibit a periodic behaviour or even perform a random walk, and the collisions may give birth to new particles following some more or less complicated rules.

This type of behaviour was first observed empirically in elementary cellular automata \#18, \#122, \#126, \#146, and \#182 \cite{Grassb, Grassa}, then \#54, \#62, \#184 \cite{Boccara}, etc. The interest for these automata stems from their dynamics that appeared neither too simple nor too chaotic, giving hope to better understand their underlying structure. In Figure~\ref{fig:particles}, we show the space-time diagrams of a sample of such automata iterated on a (uniform) random configuration.

Roughly speaking, studying particles in cellular automata requires two steps:
\begin{itemize}
 \itemsep0em
 \item identifying and describing the particles, usually as finite words;
 \item describing the particle dynamics and understanding its effect on the properties of the CA.
\end{itemize}

Historically, this study was often performed on individual or small groups of similar-looking CA, and the first step was done in a case-by-case manner. See for example \cite{Fisch, Fisch-2} for the 3-state cyclic automaton, \cite{BelitskyFerrari-1995, Belitsky-2005} for Rule $\#184$ and other automata with similar dynamics, \cite{Grassa, Eloranta-1992} for Rule $\#18$\dots Other works such as \cite{Eloranta} skip the first step and study particle dynamics in an abstract manner, deducing dynamical properties of automata by making assumptions on the dynamics of their particles. This approach was used successfully on probabilistic cellular automata simulating traffic jams, generalising standard stochastic processes such as the TASEP~\cite{GrayGriffeath}.

The first general formalism of particles in cellular automata was introduced by M. Pivato: homogeneous regions correspond to words from a subshift $\Sigma$ and particles are defects in a configuration of $\Sigma$. He developed some invariants to characterise the persistence of a defect~\cite{Pivato-2007-algebra, Pivato-2007-spectral} and he described the different possible dynamics of propagation of a defect~\cite{Pivato-2007-defect}.\sk

%\vspace{-0.3cm}

In the present work we focus on the second step first. More precisely, we are interested in how the existence of some particle set with good dynamical properties affects the typical asymptotic behaviour. Then we apply this general framework to a variety of examples, finding the sets of particles by Pivato's methods or otherwise, to explain the self-organisation that is observed experimentally. 

Let us define more formally what we mean by typical asymptotic behaviour. Starting from a $\s$-invariant measure $\mu\in\Ms(\az)$ (i.e. $\mu(\s^{-1}(U))=\mu(U)$ for all Borel set $U$), we consider the iteration of a cellular automaton $F$ on this measure:
\[
\begin{array}{rcccl}
\F:&\Ms(\az)&\longrightarrow&\Ms(\az)&\\
&\mu&\longmapsto &\F\mu&\textrm{ where }\F\mu(U)=\mu(F^{-1}(U))\textrm{ for all Borel }U.
\end{array}
\]

We then study the asymptotic properties of the sequence $(\F^t\mu)_{t\in\N}$, and particularly the set of cluster points called the \define{$\mu$-limit measures set}. Sometimes, we are only able to provide information on the $\mu$-limit set, introduced in~\cite{Kurka-Maass-2000}, which is the union of the supports of all limit measures. Equivalently, it is the set of configurations containing only patterns whose probability to appear in the space-time diagram does not tend to zero as time tends to infinity. 

When studying typical asymptotic behaviour in this sense, it is unreasonable to expect a general result since a wide variety of limit measures can be reached in the general case~\cite{Hellouin-Sablik-2013} and any nontrivial property of the $\mu$-limit set is undecidable \cite{Delacourt-2011}. That is why we consider restricted cases for the dynamics of the particles. To determine the $\mu$-limit set in some cases, P. K\r{u}rka suggests an approach based on particle weight function which assigns weights to certain words~\cite{Kurka-2003}. However, this method does not cover any case when a defect can remain in the $\mu$-limit set. Hence we aim at a more general approach, in terms of particle dynamics as well as initial measures.

One of our main motivations for this study is the class of captive cellular automata, where the local rule cannot make a colour appear if it is not already present in the neighbourhood. These automata were introduced by G. Theyssier in \cite{Theyssier-2004} for their algebraic properties, but he also noticed an interesting phenomenon: when drawing a captive cellular automaton at random (fixed alphabet and neighbourhood), most captive automata exhibited the type of self-organised behaviour described above. Any kind of general result regarding self-organisation of captive cellular automata remains a challenging open problem.

\sk 

This article is divided into two main sections, corresponding to improved versions of results previously published in conferences~\cite{Hellouin-Sablik-2011,Hellouin-Sablik-2012}. In Section~\ref{sec:particles}, we present a qualitative result generalised from~\cite{Hellouin-Sablik-2011} with an improved formalism, shorter proofs and a new application to probabilistic cellular automata (Section~\ref{sec:probabilist}). Then, in Section~\ref{sec:brown}, we refine our approach on a subclass to obtain some quantitative results. Sections~\ref{sec:walks} to \ref{sec:brownian} were published in \cite{Hellouin-Sablik-2012}; we correct some inaccuracies in the proofs and extend the study to other parameters.

\paragraph{Qualitative approach}\ In Section~\ref{sec:particles}, we prove a qualitative result: for any initial $\s$-ergodic measure $\mu$, assuming particles have good collision properties (coalescence), only particles moving in one particular direction can persist aymptotically. We introduce our own formalism of particle system in Section~\ref{sec:particlesdef} so as to be able to describe the dynamics of the particles, and Section~\ref{sec:selforg} is dedicated to the proof itself. Section~\ref{sec:defects} presents a simplified version of Pivato's formalism which is by far the simplest way to find such a particle system in most examples. 

We spend Section~\ref{sec:particlesexamples} on various examples of automata where this result can be applied:
\begin{description}
\item[Section~\ref{section.184}] we characterise the $\mu$-limit set of the ``traffic'' automaton (rule $\#184$), a simple case that may clarify the formalism. The results were known for initial Bernoulli measures~\cite{BelitskyFerrari-1995, Belitsky-2005} but our method applies for every $\s$-ergodic measure.
 \item[Section~\ref{section.cyclic}] we consider the family of $n$-cyclic cellular automata introduced in~\cite{Fisch-2,Fisch}. Using our method, we go further in the study of these simple automata: in particular, for $n=3$ or $4$, we show that the limit measure is unique and is a convex combination of Dirac measures supported by uniform configurations.
\item[Section~\ref{section.captive}] we characterise the $\mu$-limit set of all one-sided captive cellular automata. This is a first step to the study of asymptotic behaviour of captive cellular automata.
\item[Section~\ref{section.randomwalk}] last, we apply our formalism to a cellular automaton where the particles do not have a linear speed but instead perform random walks by drawing randomness from the initial measure. 
\end{description}

However, our results are not general enough to apply to defects of a sofic subshift that can have a particle-like behaviour, such as in Rule $\#18$ (see the bottom right picture in Figure~\ref{fig:particles} and \cite{Eloranta-1992}), or to more complicated particle systems such as those observed in general captive cellular automata.

Finally, in Section~\ref{sec:probabilist}, we generalise our method to probabilistic cellular automata. As an application, we partially describe limit measures of the probabilistic majority-traffic cellular automaton proposed by N. Fatès in \cite{Fates} as a candidate to solve the density classification problem. This complements the approach of \cite{Busic2013} which characterises invariant measures. Another application proposed in Section~\ref{section.Regnault-Remilia-line} presents a generalisation to the infinite line of a self stabilisation process toward a discrete line proposed in~\cite{Regnault-Remilia-2015}.

\paragraph{Quantitative approach}\ In Section~\ref{sec:brown}, we improve the previous qualitative results with a quantitative approach, considering the time evolution of some parameters when the particle dynamics are very simple. This research direction was inspired by \cite{EntryTimes}, where the authors consider the waiting time before a particle crosses the central column (called entry time). Using the same approach as in \cite{BelitskyFerrari-1995, Kurka-Maass-2000}, we show that the behaviour of these automata can be described by a random walk process (Section~\ref{sec:walks}), and we approximate this process by a Brownian motion using scale invariance (Section~\ref{sec:brownian}). Thanks to this tool, we answer negatively a conjecture proposed in~\cite{EntryTimes} by determining the correct asymptotic law for the entry time of a particle in the central column (Section~\ref{sec:entry}). We then use the same approach on various natural parameters such as the density of particles at time $t$ (Section~\ref{sec:density}) or the rate of convergence to the limit measure (Section~\ref{sec:rate}). This generalises some known results on initial Bernoulli measures from \cite{EntryTimes} and \cite{Belitsky-2005}, in particular relaxing the conditions on the initial measures. In Section~\ref{sec:extensions}, we exhibit various examples with similar dynamics on which these results apply.

In all the article, space-time diagrams were produced using the Sage mathematical software \cite{Sage} and follow the convention $\square = 0, \blacksquare = 1, \textcolor{red}{\blacksquare} = 2, \textcolor{blue}{\blacksquare} = 3 $.

%%%%
%%%%%%%%%%%%%%%%%%%%%%%%%%%%%%%%%%%%%%%%%%%
%%%%%%%%%%%%%%%%%%%%%%%%%%%%%%%%%%%%%%%%%%%
%%%%

\begin{figure}[!ht]
\begin{scriptsize}
\begin{center}
\begin{tabular}{p{0.5cm}p{0.5cm}|p{5cm}p{5cm}p{5cm}} 
&\multirow{4}{*}{\begin{sideways}{ Cases with quantitative results (Section~\ref{sec:brown})}\end{sideways} } &(-1,1)-gliders CA (Sec.~\ref{sec:walks})&Rule 184 or traffic rule (Sec.~\ref{section.184}) & 3-state cyclic CA (Sec.~\ref{section.cyclic})\\
\multirow{7}{*}{\begin{sideways}{\normalsize Cases with qualitative result (Section~\ref{sec:particles})}\end{sideways} } &&\includegraphics[width=5cm]{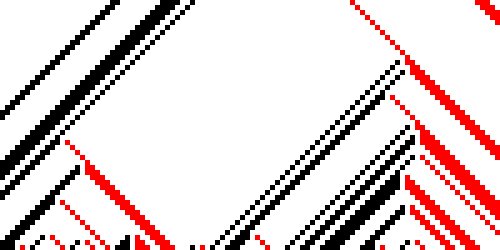}&\includegraphics[width=5cm]{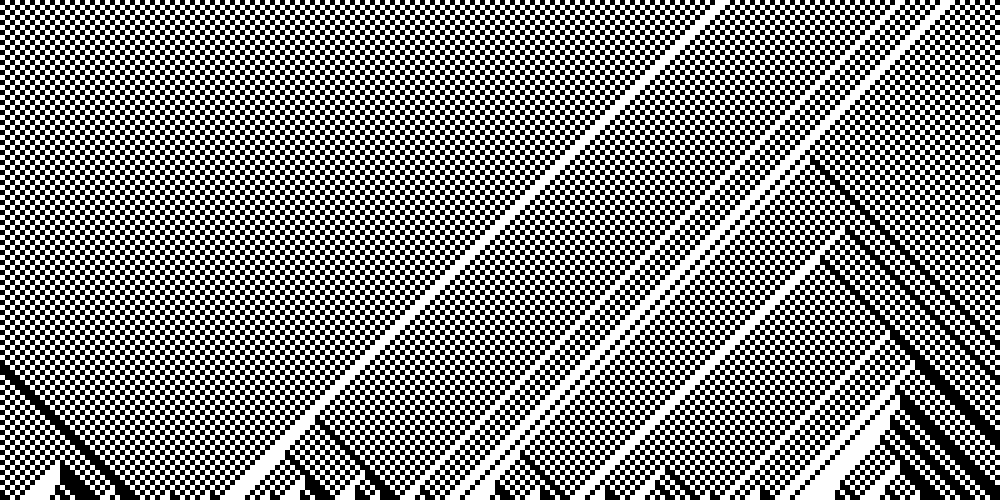}&\includegraphics[width=5cm]{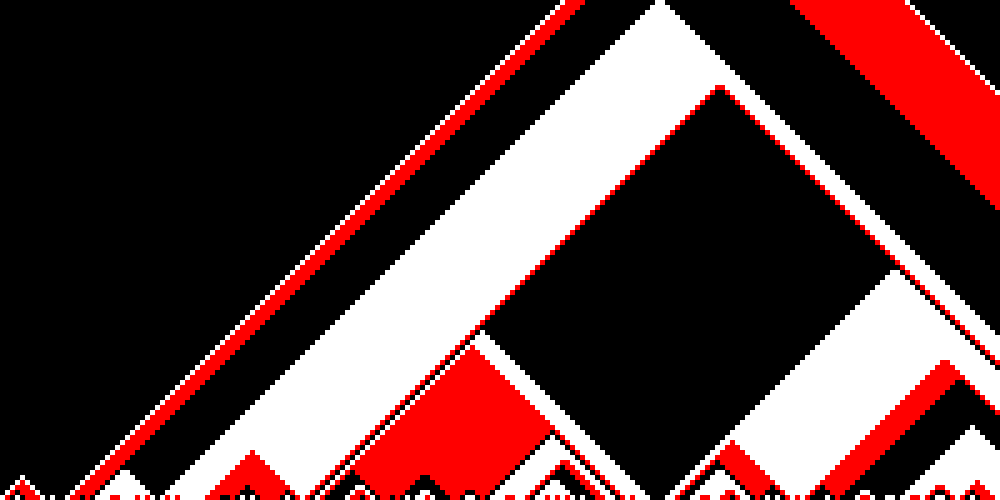}\\
&&(0,1)-gliders CA (Sec.~\ref{sec:walks})&One-sided captive CA (Sec.~\ref{section.captive})&One-sided captive CA (Sec.~\ref{section.captive})\\
&&\includegraphics[width=5cm, trim = 0 0 0 225,clip]{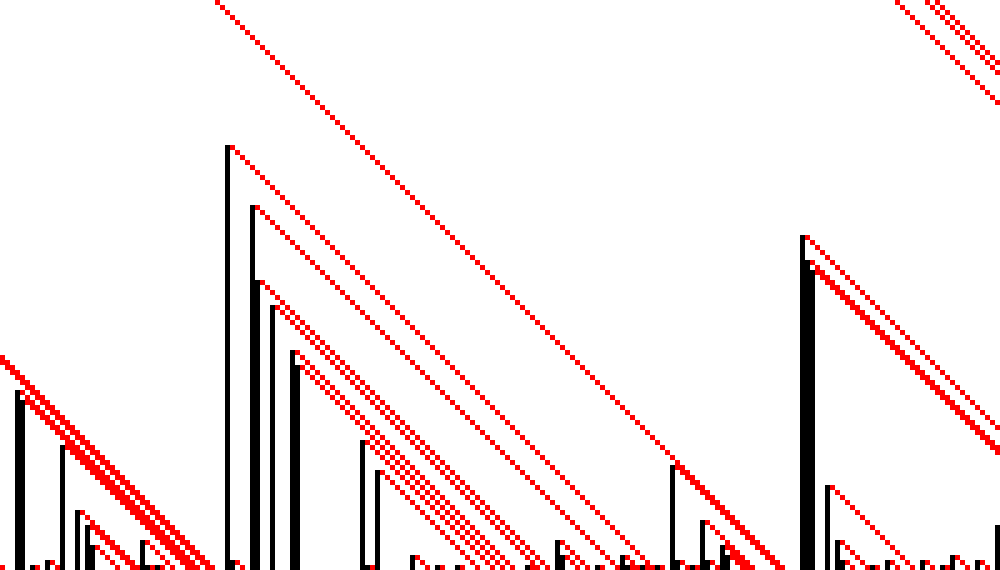}&\includegraphics[width=5cm]{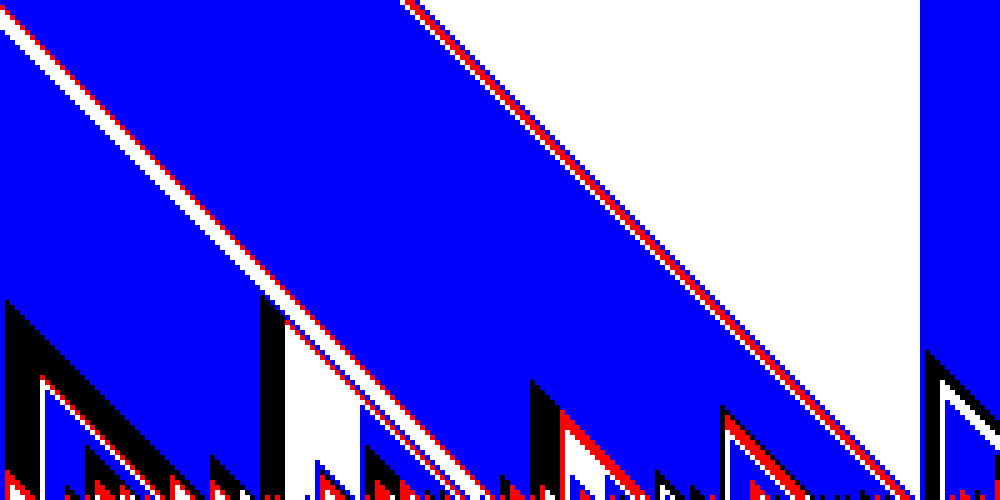}&\includegraphics[width=5cm]{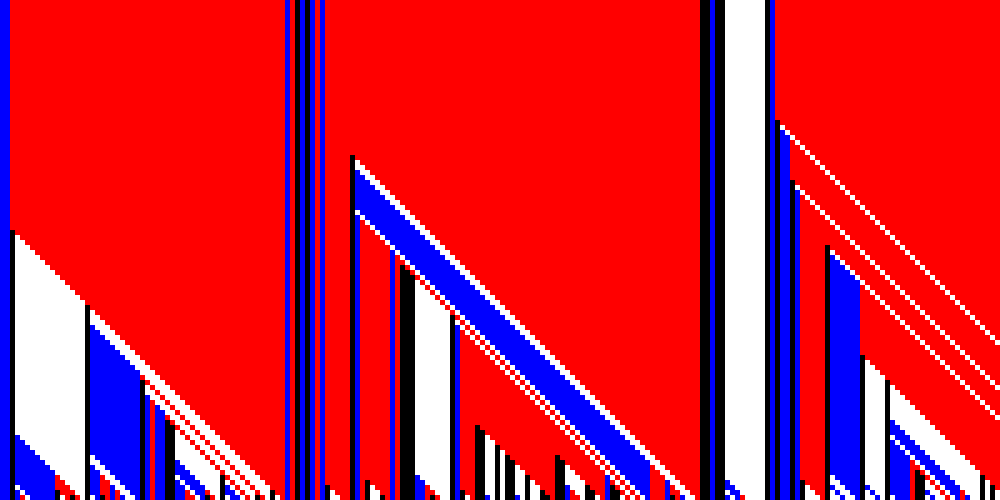}\\
\cline{2-5}
&& 4-state cyclic CA (Sec.~\ref{section.cyclic})&5-state cyclic CA (Sec.~\ref{section.cyclic})&Random walk CA (Sec.~\ref{section.randomwalk})\\
&&\includegraphics[width=5cm]{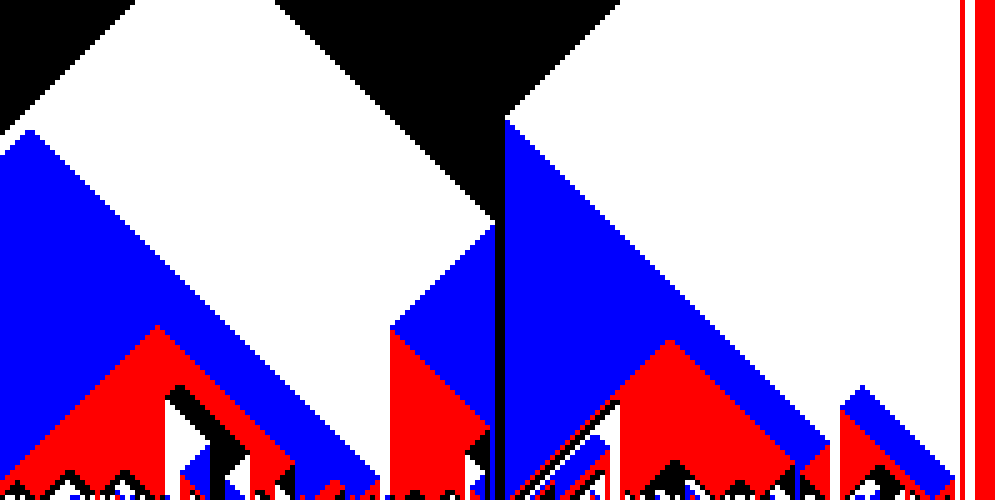}&\includegraphics[width=5cm]{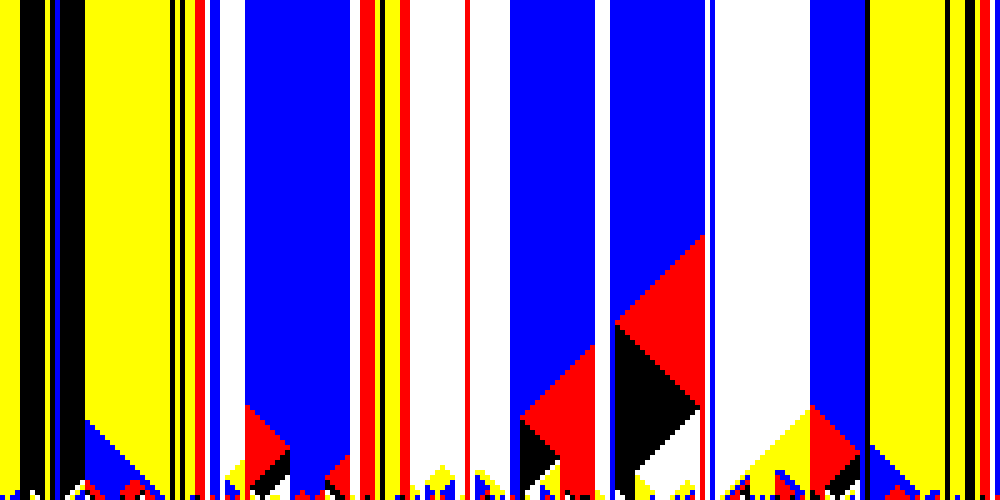}&\includegraphics[width=5cm]{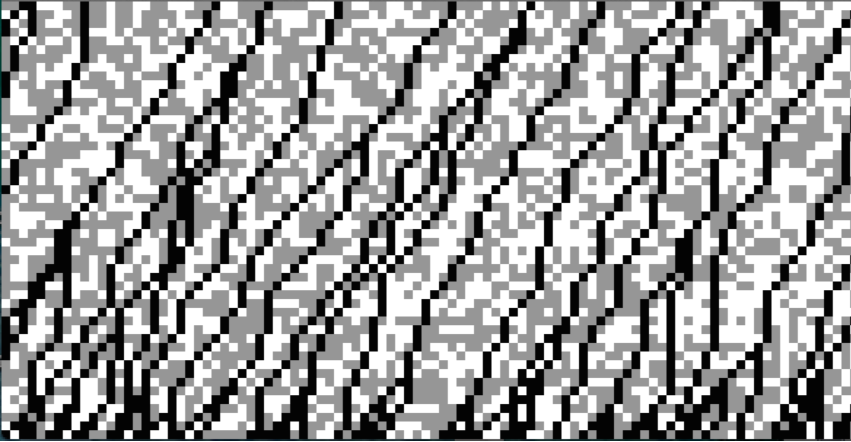}\\
\cline{2-5}
&\multirow{2}{*}{\begin{sideways}{ Prob. CA (Sec.~\ref{sec:probabilist})}\end{sideways} }&Majority-traffic PCA (Sec.~\ref{section.FatesDensity}) &Self-stabilisation of the line (Sec.~\ref{section.Regnault-Remilia-line})&Generic one-sided captive PCA\\
&&\includegraphics[width=5cm]{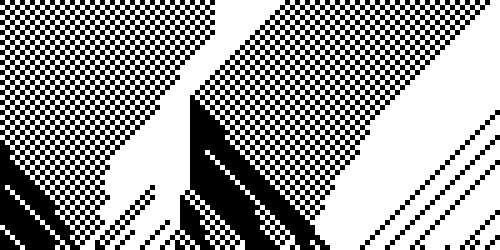}&\includegraphics[width=5cm]{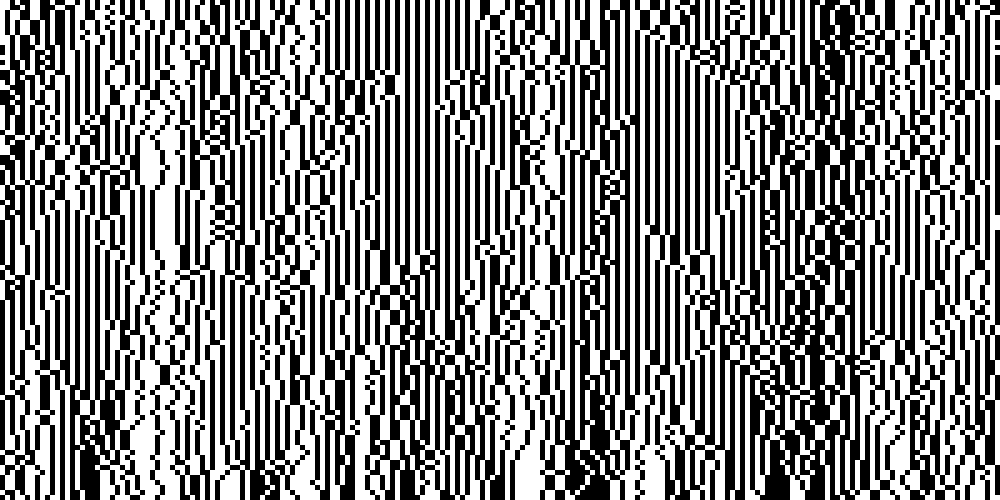}&\includegraphics[width=5cm]{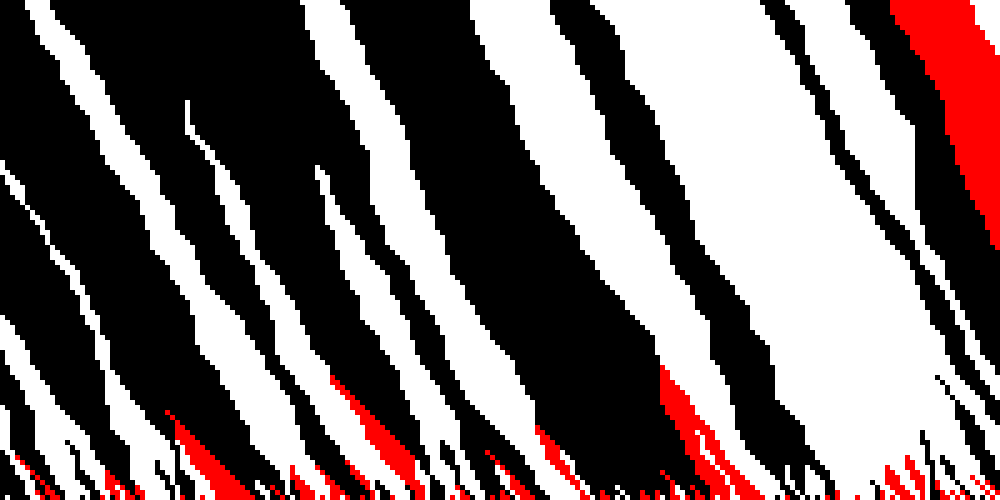}\\
\hline
\hline
&&Generic captive CA&Generic captive CA&Generic captive CA\\
\multirow{2}{*}{\begin{sideways}{\normalsize Unknown cases}\end{sideways} } &&\includegraphics[width=5cm, trim = 0 0 0 225,clip]{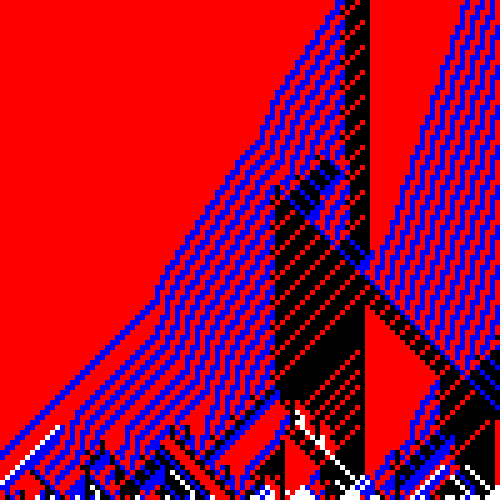}&
\includegraphics[width=5cm, trim = 0 0 0 225,clip]{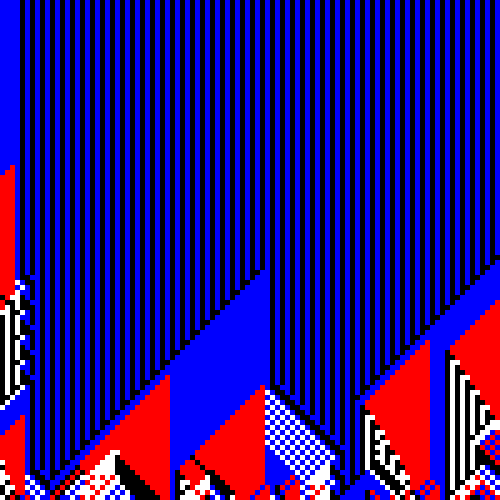}&
\includegraphics[width=5cm]{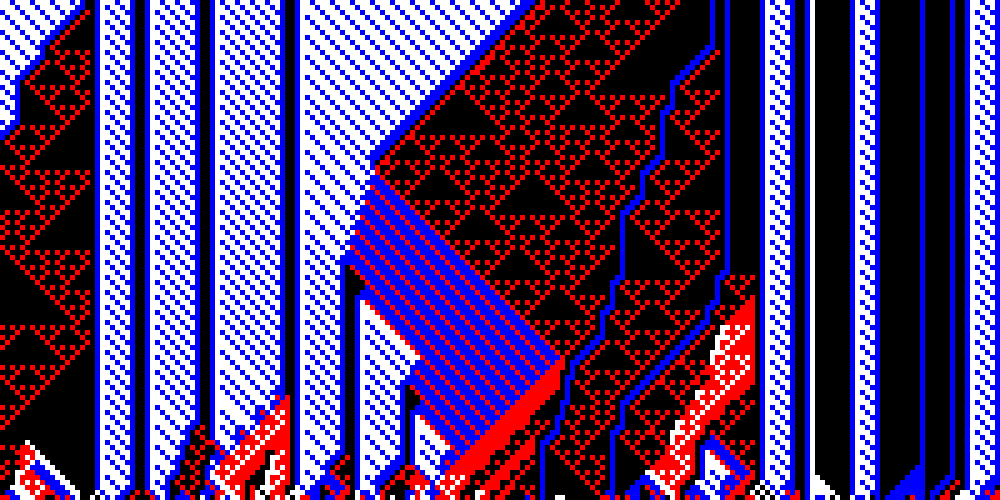}
\\
&&Generic captive CA&Generic captive CA&Rule 18\\
&&
\includegraphics[width=5cm]{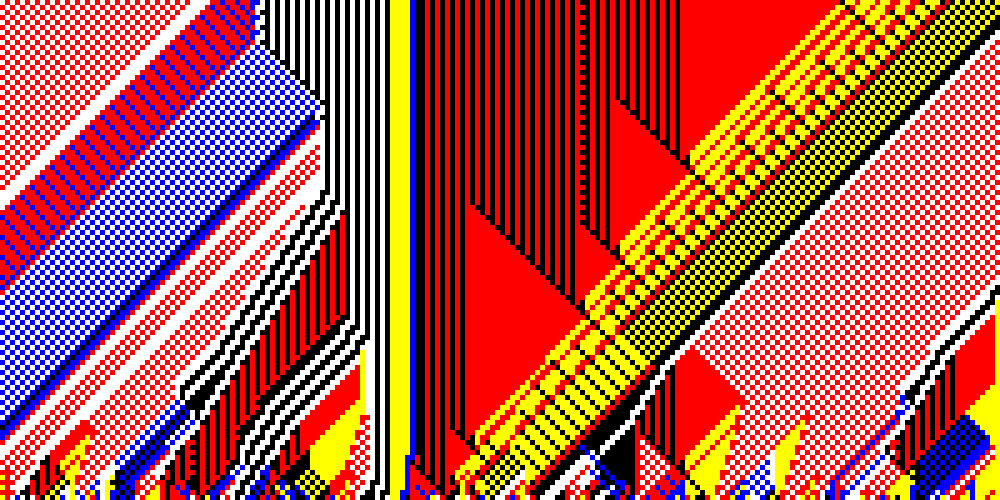}
&
\includegraphics[width=5cm]{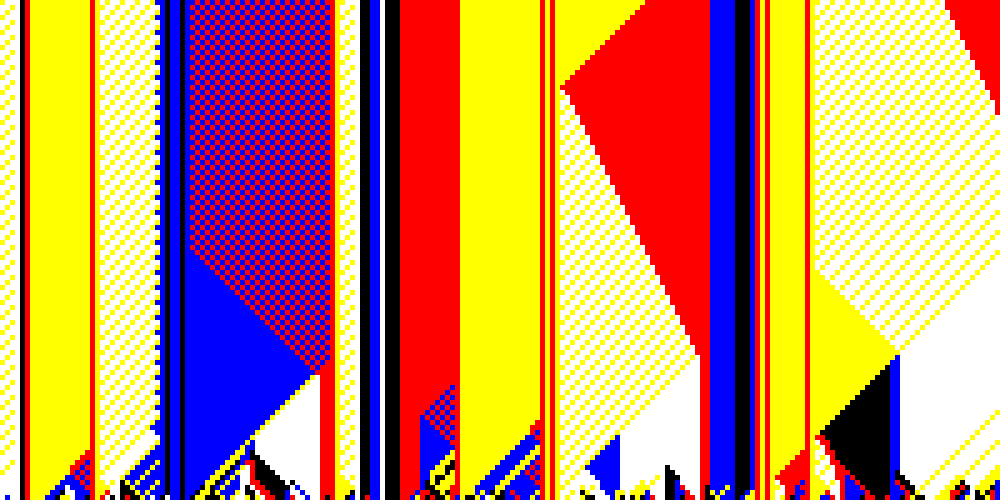}
&
\includegraphics[width=5cm]{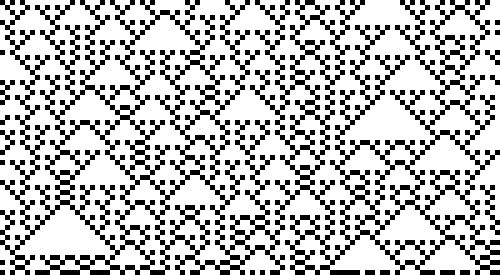}\\
\end{tabular}
\end{center}
\end{scriptsize}

\caption{Space-time diagrams of some cellular automata with particles, starting from
a configuration drawn uniformly at random. }\label{fig:particles}
\end{figure}

\section{Particle-based organisation: qualitative results}\label{sec:particles}

In this section, we take a qualitative approach to self-organisation: that is, we assume some properties on the dynamics of the particles of some cellular automaton and try to deduce properties of its $\mu$-limit measures set, with no regard to how fast this organisation takes place. 

\subsection{Particles}\label{sec:particlesdef}

\subsubsection{Definition of symbolic systems}
Given a finite alphabet $\A$, a \define{word} is a finite sequence of elements of $\A$. Denote by $\A^{\ast}=\bigcup_n\A^n$ the set of all words where $\A^0$ is the empty word $\varepsilon$. An infinite sequence indexed by $\Z$ is called a \define{configuration}. The set of configurations $\az$ is a compact set for the product topology. For a word $u\in\A^{\ast}$ the \define{cylinder} $[u]$ is the set of configurations where $u$ appears at the position $0$, and for $U\subset \A^\ast$ we have $[U] = \bigcup_{u\in U}[u]$. Cylinders are a clopen basis of the topology.

On $\az$ we define the \define{shift} map $\s(x)_i=x_{i+1}$ for all $x\in\az$ and $i\in\Z$. A \define{subshift} is a closed $\s$-invariant subset of $\az$. Equivalently, a subshift can be defined by a set of forbidden patterns $\mathcal{F}\subset \A^\ast$ as the set of configurations where no pattern of $\mathcal{F}$ appears. If $\mathcal{F}$ is finite, we call the corresponding subshift a \define{subshift of finite type} or SFT. The \define{radius} of an SFT $\Sigma$ is the smallest $\ell$ such that $\Sigma$ can be defined by a set of forbidden patterns in $\A^\ell$. The \define{language} of a subshift $\Sigma$ is defined as $\L_n(\Sigma) = \{u\in\A^n\ :\ \Sigma\cap[u]\neq \emptyset\}$ and $\L(\Sigma) = \bigcup_{n\in\N}\L_n(\Sigma)$. A SFT is \define{$\s$-transitive} if for any two patterns $u,v\in\L(\Sigma)$, there exists $w\in\A^\ast$ such that $uwv\in\L(\Sigma)$.

Given two finite alphabets $\A$ and $\B$, a \define{morphism} from $\az$ to $\B^{\Z}$ is a continuous function $\pi:\az\to\B^{\Z}$ which commutes with the shift (i.e. $\s(\pi(x))=\pi(\s(x))$ for all $x\in\az$). Equivalently a morphism can be defined by a local map $f:\A^{\mathcal{N}}\to\B$ where $\mathcal{N}\subset\Z$ is a finite set called the neighbourhood such that 
\[\pi(x)_i=f(x_{i+\mathcal N})\textrm{ for all $x\in\az$ and $i\in\Z$.}\]
The \define{radius} of $\pi$ is the minimal $r\in\N$ such that $\pi$ admits a local map with $\mathcal{N}\subset[-r,r]$. A \define{cellular automaton} is a morphism from $\az$ to itself, that is, the input and the output are defined on the same alphabet. In particular a cellular automaton can be iterated and it makes sense to study its dynamics.

\subsubsection{Particle system}

\begin{definition}[Particle system]~
Let $F:\az\to\az$ be a cellular automaton. A \define{particle system} for $F$ is a triplet $(\P, \pi, \phi)$, where:
 \begin{itemize}
 \itemsep0em
  \item $\P$ is a finite set of elements called \define{particles};
  \nomentry{$\P$}{Set of particles in a particle system}
  \item $\pi : \az \mapsto (\P\cup\{0\})^\Z$ is a morphism identifying the presence of particles at each position;
The set of positions that carry particles on $x$ is denoted by $\p_{\P, \pi}(x) = \{k\in \Z : \pi(x)_k \in \P\}$ (we omit $\P$ and $\pi$ when they are clear from the context);
  \nomentry{$\pi$}{Morphism}
\nomentry{$\p(x)$}{Set of coordinates where $x$ contains a particle}
  \item $\phi : \az\times \Z \mapsto 2^\Z$ (where $2^\Z$ denotes the set of subsets of $\Z$) is a function called the \define{update function} that describes the movement and/or offsprings of each particle after one iteration of $F$;
\end{itemize}
such that the following conditions are satisfied for all $x\in\az$ and $k\in\Z$:
\begin{description}
  \item[Locality] There is a constant $r>0$ (the \define{radius} of the system) such that $\phi(x,k)\subset[k-r,k+r]$.
  
     The particles cannot ``jump'' arbitrarily far; the radius does not depend on $x$ and $k$.
  \item[Redistribution] $\begin{array}{rl} \bigcup_{k\in\p(x)}\phi(x,k) &= \p(F(x)) \\ \bigcup_{k\notin\p(x)}\phi(x,k) &= \emptyset\end{array}$.

     The particle in $F(x)$ are exactly the offsprings of particles of $x$, and non-particles do not have offsprings.
  \item[Disjunction] $k < k' \Rightarrow \phi(x,k) = \phi(x,k') \text{ or } \max\phi(x,k) < \min\phi(x,k')$.

	Two particles either do not interact (in which case they cannot cross), or they share the same set of offsprings.
\end{description}
\nomentry{$\phi$}{Update function in a particle system}

\end{definition}

The four conditions ensure that the update function accurately describes the time evolution of the particles. Notice that since the morphism and update function are defined locally, the conditions can be checked algorithmically by simple enumeration of patterns up to a certain length.\sk

In the context of a fixed particle system for $F$, we use shorthands for the composition of the update function, defined inductively:
\[\phi^t(x,k) = \bigcup_{k'\in\phi(x,k)}\phi^{t-1}(F(x),k'),\]
and a notion of pre-image (with an abuse of notation):
\[\phi_x^{-1}(A) = \{k\in\Z\ |\ \phi(x,k) \cap A \neq \emptyset\}.\] 
If $\phi(x,k)$ is a singleton, we use ``$\phi(x,k)$'' instead of ``the only member of $\phi(x,k)$'' as an abuse of notation.

\subsubsection{Coalescence}

We postpone the discussion on how to find a particle system in a given cellular automaton to Section~\ref{sec:defects}. We now look for assumptions on the dynamics of the particles that let us deduce that some particles disappear asymptotically. Simulations suggest that this is the case when the particles are forced to collide, and that these collisions are destructive in the sense that the total number of particles decreases; thus we introduce the notion of coalescence.

\begin{definition}[Coalescence]

Let $F:\az\to\az$ be a cellular automaton, and $(\P,\pi,\phi)$ a particle system for $F$. This particle system is \define{coalescent} if, for every $x\in\az$ and $k\in\p(x)$, the particle has one of the two following behaviours:
\begin{description}
 \itemsep0em
 \item[Progression] $|\phi(x,k)| = |\phi_x^{-1}(\phi(x,k))| = 1$, and $\pi(x)_k =
\pi(F(x))_{\phi(x,k)}$

(the particle persists and its type does not change), or
\nomentry{$\prog(x)$}{Set of coordinates where $x$ contains a progressing particle}
 \item[Destructive interaction] $|\phi(x,k)| < |\phi_x^{-1}(\phi(x,k))|$

 (particles collide and generate strictly fewer particles (possibly 0); or a single particle disappears).
\nomentry{$\inter(x)$}{Set of coordinates where $x$ contains a interacting particle}
\end{description}
\end{definition}
Progressing and interacting particles of a configuration $x\in\az$ are denoted by $\prog_{\P,\pi,\phi}(x)$ and $\inter_{\P,\pi,\phi}(x)$, respectively, and $\P,\pi$ and $\phi$ are omitted when the particle system is clear from the context. $k\in\prog_{\P,\pi,\phi}(x)$ is the case when we use ``$\phi(x,k)$'' to mean ``the only member of the singleton $\phi(x,k)$''.\sk

\subsection{Probability measures and $\mu$-limit sets}

The $\mu$-limit set was introduced in~\cite{Kurka-Maass-2000} to describe the asymptotic behaviour corresponding to empirical observations. It consists in the patterns whose probability to appear does not tend to 0 when the initial point is chosen at random. To define it formally, let us introduce some notations.

Denote by $\Ms(\az)$ the set of $\s$-invariant probability measures on $\az$ (i.e. measures $\mu$ such that $\mu(\s^{-1}(U))=\mu(U)$ for any Borel set $U$). A measure is $\s$-ergodic if every $\sigma$-invariant Borel set has measure $0$ or $1$, and we denote by $\Merg(\az)$ the set of $\s$-ergodic probability measures. \cite{Walters} gives a good introduction to ergodic probability measures. 

\paragraph{Examples}~The \define{Bernoulli measure} $\lambda_{(p_a)_{a\in\A}}$ associated with a sequence $(p_a)_{a\in\A}$ of elements of $[0,1]$ whose sum is $1$ is defined by $\lambda_{(p_a)_{a\in\A}}([u])=p_{u_0}p_{u_1}\dots p_{u_{|u|-1}}$ for all $u\in\A^{\ast}$. If all the elements of $(p_a)_{a\in\A}$ have the same value $\frac{1}{|A|}$ we call it the uniform Bernoulli measure and denote it by $\lambda$. For any finite word $u\in\A^{\ast}$, define $\meas{u}$ as the unique $\s$-invariant probability measure supported by the $\s$-periodic configuration $^{\omega}u^{\omega}$ and its translations.\sk

Given a cellular automaton $F:\az\to\az$ and an initial measure $\mu\in\Ms(\az)$, we define the measure $\F\mu$ by $\F\mu(U)=\mu(F^{-1}(U))$ for any Borel set $U$. Since $F$ commutes with $\s$, one has $\F\mu\in\Ms(\az)$. Moreover if $\mu\in\Merg(\az)$, then $\F\mu\in\Merg(\az)$ as well. This allows to define the following action: 
\[
\begin{array}{rccc}
\F:&\Ms(\az)&\longrightarrow&\Ms(\az)\\
&\mu&\longmapsto &\F\mu.
\end{array}
\]

We consider the set of cluster points of the sequence $(\F^t\mu)_{t\in\N}$ called the \define{$\mu$-limit measures set} and denoted by $\V(F,\mu)$. The closure of the union of the supports of these measures is called the \define{$\mu$-limit set} and it is denoted by $\Lambda_{\mu}(F)$. Equivalently, it can be defined as the subshift
\[\Lambda_{\mu}(F)=\left\{x\in\az\ :\ \forall i,j\in \Z,\ \F^t\mu([x_{[i,j]}])\nrightarrow 0\right\}.\]

See \cite{Kurka-Maass-2000} for all basic examples. The $\mu$-limit (measures) set has been also well studied for two classes of cellular automata : automata exhibiting particle-like behaviour (\cite{Fisch-2, Belitsky-2005} and many others) and automata with an algebraic structure (\cite{Lind} and others).

\subsection{Evoution of the density of particles for coalescent systems}

Define the \define{frequency} with which the pattern $u$ appears in the configuration $x$ as
\[\freq(u,x)=\limsup_{n\to\infty}\frac{\card\{i\in[-n,n]:x_{[i,i+|u|-1]}=u\}}{2n+1}.\]
Similarly we define $\freq(S,x)$ where $S$ is a set of patterns.

We introduce the following notations for all the subsequent proofs. For $n\in\N$, let $\Ball_n$ be the set $[-n,n]\subset \Z$.
Let $F:\az\to\az$ be a cellular automaton. In the context of a fixed particle system $(\P,\pi,\phi)$, the \define{densities of particles} in a configuration $x\in\az$ are defined by:
\[\mbox{for }p\in\P,\ \D_p(x) = \freq(p,\pi(x))\quad\mbox{and}\quad \D(x) = \freq(\P,\pi(x));\]
\[\D_\prog(x) = \limsup_{t\to\infty}\frac 1{2t+1}|\prog(x)\cap\Ball_t|\quad\mbox{ and similarly for }\D_\inter(x),\]
the last two definitions applying only if the particle system is coalescent. 

For $\mu\in\Merg(\az)$, by Birkhoff’s ergodic theorem, the $\limsup$ can be replaced by a simple limit in the definition of frequency for $\mu$-almost all configurations.
This implies for example that $\D(x) = \sum_{p\in\P}\D_p(x)$ for $\mu$-almost all $x$.

First of all, the following proposition clarifies how controlling the frequency of interactions gives us information about the evolution of the density of the different kinds of particles. 

\begin{proposition}[Evolution of densities]\label{prop:technical}

Let $F : \az\to\az$ be a cellular automaton, $\mu\in\Merg(\az)$, $(\P,\pi,\phi)$ a coalescent particle system for $F$, and $r$ the radius of the update function $\phi$. Then, for $\mu$-almost all $x\in\az$:
\begin{enumerate}
\item $\D(F(x))\leq \D(x) - \frac{1}{r+1}\D_\inter(x)$;
\item $\forall p\in\P, \D_p(F(x))\leq \D_p(x) + \D_\inter(x)$.
\end{enumerate}
\end{proposition}

\begin{lemma}\label{lem:details}
\begin{enumerate}
\item For all $x$ and $k$:
\[|\phi(x,k)| + |\phi_x^{-1}(\phi(x,k))|\leq 2r+2,\]
\item which implies when $k\in\inter(x)$: \[|\phi(x,k)|\leq \frac{r}{r+1}|\phi_x^{-1}(\phi(x,k))|.\]
\end{enumerate}
\end{lemma}

\begin{proof}(of Lemma~\ref{lem:details})
Take $i\leq i'$, resp. $j\leq j'$, the extremal points of $|\phi(x,k)|$ and $|\phi_x^{-1}(\phi(x,k))|$ respectively. By locality of the update function, we have:
\[|\phi(x,k)| + |\phi_x^{-1}(\phi(x,k))|\leq (i'-i+1) + (j'-j+1) = (i'-j)+(j'-i)+2 \leq 2r+2.\]

The proof is illustrated in Figure~\ref{fig:visual}.

\begin{figure}[ht]
 \begin{tikzpicture}
  \draw (0,0) node {$x$};
  \draw (0,1) node {$F(x)$};
  \draw[black!30] (1,0) -- (13,0) (1,1) -- (13,1);
  \draw[very thick] (3,-.1) -- (3,.1) (10,-.1) -- (10,.1) (3,0) -- (10,0);
  \draw[very thick] (5,.9) -- (5,1.1) (8,.9) -- (8,1.1) (5,1) -- (8,1);
  \draw (6.5,-.5) node {$\phi_x^{-1}(\phi(x,k))$};
  \draw (6.5,1.5) node {$\phi(x,k)$};
  \draw[->] (3.1,.1) -- (7.9,.9);
  \draw (4.2,.5) node {$\leq r$};
  \draw (8.8,.5) node {$\leq r$};
  \draw[->] (9.9,.1) -- (5.1,.9);
 \end{tikzpicture}
 \caption{Visual proof that $|\phi(x,k)| + |\phi_x^{-1}(\phi(x,k))|\leq 2r+2$.}
 \label{fig:visual}
\end{figure}
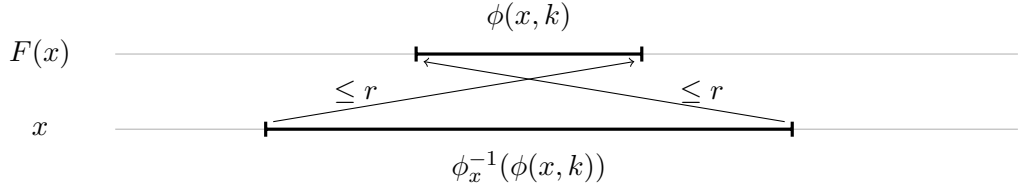
If furthermore $k\in\inter(x)$, since the particle system is coalescent, we have $|\phi(x,k)|<|\phi_x^{-1}(\phi(x,k))|$. The maximum of the ratio $\frac{|\phi(x,k)|}{|\phi_x^{-1}(\phi(x,k))|}$ is then reached on $|\phi(x,k)|=r$, $|\phi_x^{-1}(\phi(x,k))|=r+1$.
\end{proof}
We continue the proof of Proposition~\ref{prop:technical}
\begin{proof} (1) By the redistribution property of the update function, we have $\p(F(x))=\bigcup_{k\in\p(x)}\phi(x,k)$. 
Furthermore, by locality, \begin{align*}\forall x\in\az,\ \forall n\in\N,\ 
\p(F(x))\cap \Ball_n \subseteq& \bigcup_{k\in\p(x)\cap\Ball_{n+r}}\phi(x,k)\\
\subseteq&\bigcup_{k\in\prog(x)\cap\Ball_{n+r}}\phi(x,k)\quad \sqcup \bigcup_{k\in\inter(x)\cap\Ball_{n+r}}\phi(x,k),
\end{align*}

where $\sqcup$ denotes a disjoint union. The second line is obtained by coalescence: since $\p(x) = \prog(x)\sqcup \inter(x)$, 
particles in $F(x)$ are either images of progressing particles or of interacting particles. 
By disjunction:
\begin{align*}
\forall x\in\az,\ \left|\bigcup_{k\in\prog(x)\cap\Ball_{n+r}}\phi(x,k)\right| &= |\prog(x)\cap\Ball_{n+r}|\\
\mbox{and }\quad\forall x\in\az,\ \left|\bigcup_{k\in\inter(x)\cap\Ball_{n+r}}\phi(x,k)\right| &\leq
\frac{r}{r+1}\left|\phi_x^{-1}\left(\bigcup_{k\in\inter(x)\cap\Ball_{n+r}}\phi(x,k)\right)\right| \\
&\leq \frac{r}{r+1}\left|\inter(x)\cap\Ball_{n+2r}\right|.
\end{align*}

This first equality is because progressing particles are ``one-to-one''. The second inequality is by Lemma~\ref{lem:details} and by locality. It follows:
\[\forall x\in\az,\ |\p(F(x))\cap \Ball_n| \leq |\prog(x)\cap\Ball_{n+r}| + \frac{r}{r+1}\left|\inter(x)\cap\Ball_{n+2r}\right|.\]
Then, passing to the limit:
\[\textrm{For $\mu$-almost all } x\in\az,\ \D(F(x))\leq \D_\prog(x)+\frac{r}{r+1}\D_\inter(x) = \D(x)-\frac1{r+1}\D_\inter(x).\]

(2) Similarly, for any particle $p\in\P$, one has for all $x\in\az$ and $n\in\N$:
\[\{k\in\Ball_n\ |\ \pi(F(x))_k=p\} \subseteq
\bigcup_{k\in\p(x)\cap\Ball_{n+r}}\phi(x,k)\quad\mbox{(locality)}.\]
For $k\in\prog(x)$, if $\pi(F(x))_{\phi(x,k)} = p$, then by definition of coalescence $\pi(x)_k = p$. For $\mu$-almost all $x$, using $\p(x) = \prog(x)\sqcup \inter(x)$, we conclude that $\D_p(F(x))\leq \D_p(x)+\D_{inter}(x)$ by passing to the limit.
\end{proof}

\subsection{A particle-based self-organisation result}\label{sec:selforg}

We state our main result. A simple version (Corollary~\ref{cor:MainResult}) states that in a coalescent particle system with a $\s$-ergodic initial measure, if all particles can be assigned a speed, then only particles with one fixed speed may survive asymptotically. The more general result is designed to handle more difficult cases such as particles performing random walks, as in the last example of Section~\ref{sec:particlesexamples}.

\begin{definition}[Clashing]
Let $F:\az\to\az$ be a cellular automaton, $(\P,\pi,\phi)$ a coalescent particle system for $F$, and $\P_1$ and $\P_2$ two subsets of $\P$. We say that $\P_1$ \define{clashes with} $\P_2$ $\mu$-almost surely if, for every $n\in\N^\ast$ and $\mu$-almost all $x\in\az$,
\[\pi(x)_0\in\P_1 \mbox{ and }\pi(x)_n\in\P_2 \Longrightarrow \exists t\in\N,\phi^t(x,0)\in\inter(F^t(x))\mbox{ or }\phi^t(x,n)\in\inter(F^t(x))\]
\end{definition}
The abuse of notation in the last line is justified by the fact that, if the images $\phi^{t}(x,k)$ ($k=0,n$) are not in interaction for all $t'<t$, then $\phi^{t}(x,k)$ is still a singleton.

The intuition behind clashing particles is the following: if two clashing particles are present with positive frequency, then at least one of them end up almost surely in interaction with positive frequency, decreasing the global frequency of particles. This is why they cannot both persist asymptotically. Note that clashing is oriented left to right: intuitively, particles with speed $+1$ clash with particles with speed $-1$, but the converse is not true.

\begin{theorem}[Main qualitative result]\label{prop:MainResult}

Let $F:\az\to\az$ be a cellular automaton, $\mu$ an initial $\s$-ergodic measure and $(\P,\pi,\phi)$ a coalescent particle system for $F$ where $\P$ can be partitioned into sets $\P_1\dots \P_n$ such that, for every $i<j$, $\P_i$ clashes with $\P_j$ $\mu$-almost surely.

Then:
\begin{enumerate}
\item All particles appearing in the $\mu$-limit set belong to the same subset, i.e.
\[\exists i\in[1,n],\ \forall p\in\P,\ p\in\L(\pi(\Lambda_\mu(F)))\Rightarrow p\in\P_i.\]

\item If furthermore there exists a $j$ such that $\P_j$ clashes with itself $\mu$-almost surely, then this subset of particles does not appear in the $\mu$-limit set, i.e.
\[\forall p\in\P,\ p\in\L(\pi(\Lambda_\mu(F)))\Rightarrow p\notin\P_j.\]
\end{enumerate}
\end{theorem}

We introduce the notion of \define{speed} which is less general but easier to handle than the notion of clashing.

\begin{definition}[Speed]

Let $F$ be a cellular automaton and $(\P,\pi,\phi)$ be a particle system for $F$.

 A particle $p\in\P$ has \define{speed $\bm{v\in\Z}$} if for any configuration $x\in\az$ and $k\in\Z$ such that $\pi(x)_k=p$, we have one of the following:
 \begin{description}
  \itemsep0em
  \item[Eventual interaction] $\exists t, \phi^t(x,k) \in\inter(F^t(x))$;
  \item[Progression at speed $v$] $\forall t, \phi^t(x,k)\in\prog(F^t(x))$ and $\frac{\phi^t(x,k)-k}t\underset{t\to\infty}\to v$.
 \end{description}
\end{definition}

\begin{corollary}[Version with speedy particles]\label{cor:MainResult}

Let $F:\az\to\az$ be a cellular automaton, $\mu$ an initial $\s$-ergodic measure and $(\P,\pi,\phi)$ a coalescent
particle system for $F$.

If each particle $p\in\P$ has speed $v_p\in\R$,then there is a speed  $v\in\R$ such that:
\[\forall p\in\P, p\in\L(\pi(\Lambda_\mu(F)))\Rightarrow v_p=v.\]
\end{corollary}

\begin{proof}[Proof of Theorem~\ref{prop:MainResult}]
For the first point, Let $i=1,\, j=2$ for clarity and let $p_1\in\P_1, p_2\in\P_2$ be two particles. We show that they cannot both appear in the $\mu$-limit set.

First we study the behaviour of the sequences of density of particles. For all $x\in\az$, by Proposition~\ref{prop:technical}(1), $(\D(F^t(x)))_{t\in\N}$ is a decreasing sequence of positive reals and admits a limit $d_\infty(x)$. In particular $\D_\inter(x)\to 0$. Applying Birkhoff's theorem to $\pi_\ast\F^t\mu$ for any $t$, we get that $\D(F^t(x)) = \pi_\ast\F^t\mu([\P])$ for $\mu$-almost all~$x$ (recall that $[\P] = \bigcup_{p\in\P}[p]$). In particular there is a real $d_\infty$ such that $d_\infty(x) = d_\infty$ for $\mu$-almost all $x$.

For $x\in\az$, we define $\D_{\P_i}(x) = \freq(\P_i, \pi(x))$; we prove that this sequence also admits a limit. By Proposition~\ref{prop:technical}(2), we have:
\[\mbox{For }i\in\{1,2\},\ \sup_{n\in\N}|\D_{\P_i}(F^{t+n}(x))-\D_{\P_i}(F^t(x))|\leq \sum_{n=0}^\infty \D_\inter(F^{t+n}(x)).\]

To prove that $(\D_{\P_i}(F^t(x)))_{t\in\N}$ is a Cauchy sequence, we need to show that $\sum_{t\in\N}\D_\inter(F^t(x))<+\infty$. By Proposition~\ref{prop:technical}(1), we have:
\[\sum_{t\in\N}\D_\inter(F^t(x)) \leq (r+1)\left(\sum_{t\in\N}\D(F^t(x))-\D(F^{t+1}(x))\right)\leq (r+1)(\D(x)-d_\infty(x)) < +\infty.\]
Thus $(\D_{\P_i}(F^t(x)))_{t\in\N}$ is a Cauchy sequence and admits a limit $d_i(x)\neq 0$. Using again Birkhoff's theorem, we have that $(\D_{\P_i}(F^t(x)))_{t\in\N} = (\pi_\ast\F^t\mu([\P_i]))_{t\in\N}$ for $\mu$-almost all $x$, and therefore there is a real $d_i$ such that $d_i(x) = d_i$ for $\mu$-almost all $x$.\sk

Assume that $p_i\in\L(\pi(\Lambda_\mu(F)))$ for $i=1,2$. This implies $d_i>0$ for $i=1,2$. Since clashing particles generate interactions, we show that this contradicts the fact that $\sum \D_\inter(F^t(x))<+\infty$ for all $x$.

Fix $\varepsilon < \frac{d_1\cdot d_2}{r+3}$ and $T$ large enough such that for $t\geq T,\
\pi_\ast\F^t\mu([\P])-d_\infty<\varepsilon$
and $|\pi_\ast\F^t\mu([\P_i])-d_i| < \varepsilon$ for $i\in\{1,2\}$.
%Denote $V_k = \{x\in\az\ :\ \pi(x)_0 = p_1, \pi(x)_k = p_2\}$. 
By Birkhoff's ergodic theorem applied on $\pi_\ast\F^T\mu$, we have:
\[\frac 1K\sum_{k=0}^K\pi_\ast\F^T\mu\left([p_1]_0\cap[p_2]_k\right) \underset{K\to\infty}\longrightarrow
\pi_\ast\F^T\mu([p_1])\cdot\pi_\ast\F^T\mu([p_2]).\]
Note that $[p_1]_0\cap[p_2]_k$ are words containing clashing particles positioned so that they will generate an interaction. We have $\pi_\ast\F^T\mu([p_1])\cdot\pi_\ast\F^T\mu([p_2]) \geq (d_1-\varepsilon)\cdot (d_2-\varepsilon)\geq d_1\cdot d_2-2\varepsilon$. By Birkhoff's theorem, this means that for $\mu$-almost all $x\in\az$, words belonging in $\bigcup_k V_k$ where $V_k = p_1(\P\cup\{0\})^{k-1}p_2\subset(\P\cup\{0\})^\ast$
have frequency at least $d_1\cdot d_2-2\varepsilon$ in $\pi F^T(x)$.

Since $\P_1$ and $\P_2$ clash $\mu$-almost surely, any occurrence of $V_k$ yields a future interaction: that is, $\freq\left(\bigcup_k V_k,\pi F^T(x)\right) \leq \sum_{t=T}^\infty \D_\inter(F^t(x))$. We show the contradiction:
\begin{align*}
\textrm{For $\mu$-almost all } x\in\az,\ \D(F^T(x))-d_\infty &\geq \frac 1{r+1} \sum_{t=T}^\infty \D_\inter(F^t(x))&\mbox{Proposition~\ref{prop:technical}(i)}\\
                    &\geq \frac 1{r+1}(d_1\cdot d_2-2\varepsilon)>\varepsilon,
\end{align*}
which is a contradiction with the definition of $\varepsilon$. To prove the second point, apply the same proof to two particles in $\P_j$.
\end{proof}

\begin{proof}[Proof of Corollary~\ref{cor:MainResult}]
Consider the set of speeds $\{v_p\ :\ p\in\P\}$ and order it as $v_1>v_2> \dots>v_n$. Now partition the set of particles into $(\P_{v_i})_{0\leq i\leq n}$ where $\P_{v_i}$ is the set of particles with speed $v_i$, and apply the Theorem~\ref{prop:MainResult}.

We check the hypothesis of Theorem~\ref{prop:MainResult}: for any $i<j$, $\P_{v_i}$ clashes with $\P_{v_j}$ $\mu$-almost surely. Let $p_i\in\P_{v_i}$ and $p_j\in\P_{v_j}$, and $x\in\az$ such that $\pi(x)_0=p_i$ and $\pi(x)_n=p_j$ for some $n\in\N^\ast$. If both particles satisfy the second property in the definition of speed (Progression at speed $v$), then for some $t$ large enough we have $\phi^t(x,0)>\phi^t(x,n)$, which is forbidden by coalescence since two particles in progression cannot cross. Thus at some time $t$ we have either $\phi^t(x,0)\in\inter(F^t(x))$ or $\phi^t(x,n)\in\inter(F^t(x))$.
\end{proof}

\subsection{Pivato's defect formalism}\label{sec:defects}

Before giving a series of examples where this result can be used to describe the typical asymptotic behaviour of a cellular automaton, we present the formalism introduced by Pivato in~\cite{Pivato-2007-algebra,Pivato-2007-spectral} that defines particles as defects with respect to a $F$-invariant subshift $\Sigma$. Indeed, this formalism gives us an easier way to find the particle systems in our examples.

Intuitively, the $F$-invariant subshift describes the homogeneous regions that persist under the action of $F$ in the space-time diagram, and defects are the borders between these regions. This allows us to define $\P$ and $\pi$ in a way that corresponds to the intuition, even though it gives no information on the dynamics of the particles (the update function $\phi$).

\subsubsection{Defects}

For a cellular automaton $F$, consider $\Sigma$ a $F$-invariant subshift. The \define{defect field} of $x\in\az$ with
respect to $\Sigma$ is defined as:
\[\mathcal{F}^{\Sigma}_x:\begin{array}{ccl}
			      \Z&\to&\N\cup\{\infty\}\\
                              k&\mapsto&\displaystyle \max \left\{n\in \N : x_{k+[-\lfloor\frac {n-1}2\rfloor, \lceil\frac{n-1}2\rceil]}\in \L_n(\Sigma)\right\}
                            \end{array}
,\]
where the result is possibly $0$ or $\infty$ if the set is empty or infinite. Intuitively, this function returns the size of the largest word admissible for $\Sigma$ centred on the cell $k$. A \define{defect} in a configuration $x$ relative to $\Sigma$ is a local minimum of $\mathcal{F}^{\Sigma}_x$. Then the interval $[k,l]$ between two defects forms a homogeneous region in the sense that $x_{[k+1,l]}\in \mathcal L(\Sigma)$.\sk

However, it is not true that we can always make a correspondence between defects and a finite set of words (forbidden patterns), so as to obtain a finite set of particles and a morphism. This is the case only when the set of forbidden patterns is finite, that is, when $\Sigma$ is a SFT. In this case, a defect corresponds to the centre of the occurrence of a forbidden word. This is a limitation of our result.\sk

The examples given in Figure~\ref{fig:particles} suggest that defects can usually be classified using one of these approaches:
\vspace{-.4\baselineskip}
\begin{itemize}
 \itemsep0em
 \item Regions correspond to different subshifts and defects behave according to their surrounding regions
(\define{interfaces} - e.g. cyclic automaton);
 \item Regions correspond to the same periodic subshift and defects correspond to a ``phase change''
(\define{dislocations} - e.g. rule 184 automaton).
\end{itemize}

\subsubsection{Interfaces}

Assume that $\Sigma$ is a SFT and can be decomposed as a disjoint union $\Sigma_1 \sqcup \dots \sqcup \Sigma_n$ of $F$-invariant $\s$-transitive SFTs (the \define{domains}). Intuitively, the region between two defects belongs to the language of (at least) one of the domains; we classify each defect according to which domain the regions surrounding it on the left and on the right correspond to. Since each domain is $F$-invariant, this classification is conserved under the action of $F$ for non-interacting defects.

Formally, since the different domains $(\Sigma_k)_{k\in[1,n]}$ are disjoint SFTs, there is a length $\alpha>0$ such that $(\L_{\alpha}(\Sigma_k))_{k\in[1,n]}$ are disjoint (if two subshifts share arbitrarily long words, they share a configuration by closure). In particular, if $u\in\L_{\alpha}(\Sigma)$, then there is a unique $k$ such that $u\in\L(\Sigma_k)$: we say that $u$ \define{belongs to the domain $k$}. Thus, for a given configuration, we can assign a choice of a domain to each homogeneous region between two consecutive defects, and this choice is unique if this region is larger than $\alpha$ cells.

\begin{figure}[!ht]
\begin{center}
\begin{tikzpicture}
 \foreach \x/\colour/\k/\f in
{-1/black/1/7,0/black/1/5,1/black/1/3,2/black/1/1,3/white/0/2,4/white/0/4,5/white/0/5,6/white/0/3,7/white/0/1,8/gray/2/2
,9/gray/2/4,10/gray/2/5,11/gray/2/3, 12/gray/2/1, 13/white/0/2, 14/white/0/4, 15/white/0/6}
   {
   \filldraw[draw = black] [fill=\colour!80] (0.7*\x, 0) rectangle (0.7*\x+0.7,0.7);
   \draw (0.7*\x+0.35, 0.95) node {\small \k};
   \draw (0.7*\x+0.35, -0.25) node {\small \f};
   }
 \draw [draw = red,thick] (2.1,-0.4) -- (2.1,1);
 \draw [draw = red, thick] (5.6,-0.4) -- (5.6,1); 
\draw [draw = red, thick] (9.1,-0.4) -- (9.1,1);
\draw[red] (2.1, 1.4) node {$1-0$};
\draw[red] (5.8, 1.4) node {$0-2$};
\draw[red] (9.3, 1.4) node {$2-0$};
 \draw (-1.4, -0.25) node {$\mathcal F_a(x)$};
 \draw (-1.8, 0.35) node {$x$};
 \draw (-1.4, 1) node {domain};
 \draw (-1.1, 0.35) node {\dots};
 \draw (11.6, 0.35) node {\dots};
 \draw [pattern = north east lines] (1.4,0) rectangle (2.1,0.7);
 \draw [pattern = north east lines] (4.9,0) rectangle (5.6,0.7);
  \draw [pattern = north east lines] (8.4,0) rectangle (9.1,0.7);
\end{tikzpicture}

\caption{Interfaces between monochromatic domains, marked by slanted patterns. To each interface corresponds a domain change, marked by a red line.}\label{fig:defect}
\end{center}
\end{figure}
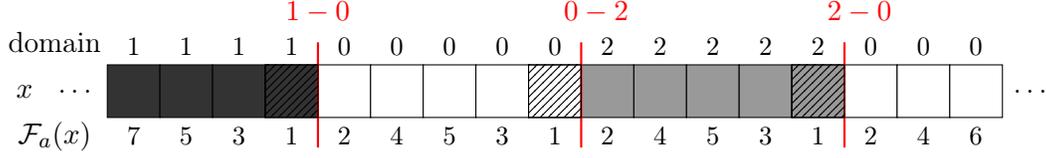

Defects relative to such an SFT are called \define{interface defects} and can be classified according to the domain of the surrounding regions. Let $\P = \{p_{ij}\ :\ (i,j)\in[1,n]^2\}$ be the set of particles. Define the morphism $\pi : \az \to (\P\cup\{0\})^\Z$ of radius $\max(\lceil r/2\rceil,\alpha)$, where $r$ is the radius of $\Sigma$, in the following way. For $x\in\az$ and $k\in\Z$:

\begin{itemize}
 \itemsep0em
 \item if $x_{k+[-\lfloor\frac r2\rfloor, \lceil\frac r2\rceil]}\in\L(\Sigma)$, then $\pi(x)_k=0$;
 \item else, let $\left\{\begin{array}{l}
                  u_1 = x_{[k-m,k]}$ where $m = \max\{\ell\leq \alpha\ :\ x_{[k-\ell,k]}\in\L(\Sigma)\}\\
                  u_2 = x_{[k+1,k+m]}$ where $m = \max\{\ell\leq \alpha\ :\ x_{[k+1,k+\ell]}\in\L(\Sigma)\}\\
                  d_i \mbox{ a domain to which }u_i\mbox{ belongs }(i\in\{1,2\})
                 \end{array}\right.$
                 
and put $\pi(x)_k=p_{d_1d_2}$.
\end{itemize}

The domain choice (choice of $d_i$) is unique when domains contain at least $\alpha$ cells; otherwise, the choice between the possible $d_i$ is arbitrary, or fixed beforehand. Notice that the first check requires radius at least $\lceil\frac r2\rceil$ and the second check requires radius at least $\alpha$.

\subsubsection{Dislocations}

Contrary to interface defects that mark a change between languages of different SFT, \emph{dislocation defects} mark a ``change of phase'' inside a single SFT. 

Let $\Sigma$ be a $\s$-transitive SFT of order $r>1$. We say that $\Sigma$ is \define{$P$-periodic} if there exists a
partition $V_1,\dots,V_P$ of $\L_{r-1}(\Sigma)$ such that 
\[a_1\cdots a_r\in\L_r(\Sigma)\quad\Leftrightarrow \quad \exists i\in\Z /P\Z,\ 
a_1\cdots a_{r-1}\in V_i\mbox{ and }a_2 \cdots a_r\in V_{i+1}.\] 
The \define{period} of $\Sigma$ is the maximal $P\in\N$ such that $\Sigma$ is $P$-periodic. For example, the orbit of a finite word $u\in\A^\ast$, defined as $\{\s^k(\mathstrut^\infty u^\infty) : k\in\Z\}$ is a periodic SFT of period at most $|u|$.

We thus associate to each $x\in\Sigma$ its \define{phase} $\varphi(x) \in\Z/P\Z$ such that $x_{[0,r-2]}\in
V_{\varphi(x)}$. Obviously, $\varphi(\s^k(x)) = \varphi(x)+k\mod p$. For $x\in\az$, we say that an homogeneous region
$[a,b]$ (i.e. a region such that $x_{[a,b]}\in\Sigma$) is \define{in phase $k$} if $\exists y\in\Sigma, \varphi(y)=k,
x_{[a,b]}=y_{[a,b]}$. If $b-a>r-2$, the phase of a region is unique and means $x_{[a,a+r-2]}\in V_{k+a\mod p}$.

\begin{figure}[!ht]
\begin{center}
\begin{tikzpicture}
 \foreach \x/\colour/\k/\f in {-1/black/1/7,
0/white/1/5,1/black/1/3,2/white/1/1,3/white/0/2,4/black/0/4,5/white/0/6,6/black/0/5,7/white/0/3,8/black/0/1,9/black/1/2,
10/white/1/4,11/black/1/5,12/white/1/3,13/black/1/1, 14/black/0/2, 15/white/0/4, 16/black/0/6}
   {
   \filldraw[draw = black] [fill=\colour!80] (0.7*\x, 0) rectangle (0.7*\x+0.7,0.7);
   \draw (0.7*\x+0.35, 0.95) node {\small \k};
   \draw (0.7*\x+0.35, -0.25) node {\small \f};
   }
 \draw [draw = red,thick] (2.1,-0.4) -- (2.1,1.1);
 \draw[red] (2.1, 1.4) node {$1-0$};
 \draw [draw = red, thick] (6.3,-0.4) -- (6.3,1.1);
 \draw[red] (6.3, 1.4) node {$0-1$};
 \draw [draw = red, thick] (9.8,-0.4) -- (9.8,1.1);
 \draw[red] (9.8, 1.4) node {$0-1$};
 \draw (-1.4, -0.25) node {$\mathcal{F}_a(x)$};
 \draw (-1.6, .35) node {$x$};
 \draw (-1.4, 1) node {phase};
 \draw (-1.1, 0.35) node {\dots};
 \draw (12.3, 0.35) node {\dots};
 \draw [pattern = north east lines] (1.4,0) rectangle (2.1,0.7);
 \draw [pattern = north east lines] (5.6,0) rectangle (6.3,0.7);
 \draw [pattern = north east lines] (9.1,0) rectangle (9.8,0.7);
\end{tikzpicture}
\end{center}
\caption{Dislocations in the chequerboard subshift ($P=2$), marked by slanted patterns. Red lines show the visual
intuition of a change of phase, with the surrounding local phases.}
\label{fig:dislo}
\end{figure}
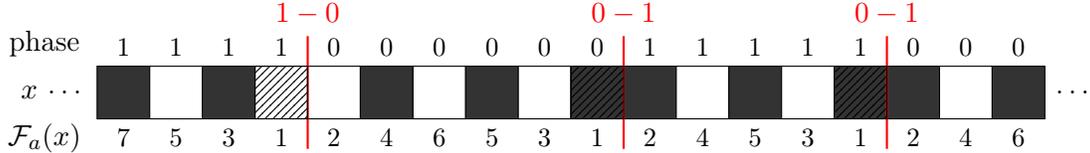

As we can see in Figure~\ref{fig:dislo}, the finite word corresponding to a defect (here $00$ or $11$) does not depend only on the phase of the surrounding region but also on the position of the defect. More precisely, since $\varphi(\s(x)) = \varphi(x)+1$, a defect in position $j$ with a region in phase $f$ to its left and a defect in position $0$ with a region in phase $f+j\mod P$ to its left ``observe'' the same finite word to their left.

Therefore, we define for each defect its \define{local phases}. Assume a defect is in position $j$ surrounded by homogeneous regions $[i,j]$ and $[j,k]$ in phase $\varphi_\ell$ and $\varphi_r$, respectively. Then its \emph{left local phase} (resp. \emph{right local phase}) is $\varphi_\ell+j\mod P$, resp. $\varphi_r+j\mod P$.

Now we classify the defects according to the local phase of the surrounding regions. Let $\P = \{p_{ij}\ :\ (i,j)\in\Z/P\Z^2\}$ be the set of particles. Since defects correspond to the centre of occurrences of forbidden words and the phase of a region can be locally distinguished, the morphism $\pi : \az \to (\P\cup\{0\})$ of order $2r-2$ is defined exactly as in the interface case. The choice of local phase is unique if the region is larger than $r-1$ cells.\sk

In the general case, those two formalisms can be mixed by fixing a decomposition $\Sigma = \bigsqcup_{i\in\A}\Sigma_i$ where some of the $\Sigma_i$ have nonzero periods. We can classify defects according to the domains and local phase of the surrounding regions in a similar manner. Except for the arbitrary choices for small regions, obtaining the set of particles and the morphism from the SFT decomposition can be done in an automatic way.

\subsection{Examples}\label{sec:particlesexamples}

\subsubsection{Rule 184}\label{section.184}

We consider the rule $\#184$ or ``traffic'' automaton $F_{184}:\{0,1\}^{\Z}\to\{0,1\}^{\Z}$ defined by the following local rule: $f_{184}(x_{-1}x_0x_1)=1$ if and only if $x_0x_1=11$ or $x_{-1}x_0=10$.

The time evolution of this automaton can be seen as a road where the symbol $1$ represent vehicles and the symbol $0$ an empty space. The vehicles move forward if the cell in front of them is empty and stay put otherwise. In this context, the rule $\#184$ has been very well studied, especially in the case of initial Bernoulli measures \cite{BelitskyFerrari-1995, Belitsky-2005}. We use this example mostly as a simple case to better understand the formalism, although our method has the advantage to hold for more general probability measures.

\begin{proposition}
Let $F_{184}$ be the traffic automaton and $\mu \in \Merg$. Then:
\begin{align*}
 \mu([00])>\mu([11])&\Rightarrow 11\notin \L(\Lambda_\mu(F_{184}));\\
 \mu([00])<\mu([11])&\Rightarrow 00\notin \L(\Lambda_\mu(F_{184}));\\
 \mu([00])=\mu([11])&\Rightarrow F_{184\ast}^t\mu \to \meas{01}.
\end{align*}
\end{proposition}

\begin{figure}[!ht]
\begin{center}
\includegraphics[width = 7cm]{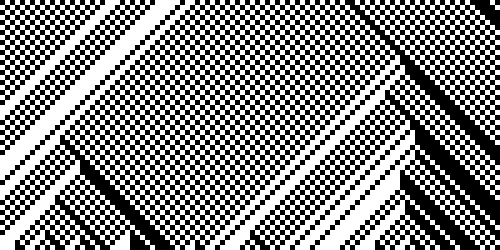}$\begin{array}{c}\to\vspace{3cm} \end{array}$\includegraphics[width =
7cm]{Particles/Factgr.png}\vspace{-1.2cm}
\end{center}
\caption{Particle system for the traffic automaton.}
\label{fig:Factors}
\end{figure}

\begin{proof}
We consider the chequerboard SFT $\Sigma = \{\per{(01)},\per{(10)}\}$, which is $2$-periodic and $F_{184}$-invariant. Using the dislocation formalism, we define the phases $\varphi(\per{(01)}) = 0$ and $\varphi(\per{(10)}) = 1$, obtaining a set of particles defined by their local phases $\{p_{01}, p_{10}\}$. The corresponding morphism of order $r=2$ is defined by the local rule:
\[\begin{array}{ccc}
         00&\to& p_{01}\\
         11&\to& p_{10}\\
         \mbox{otherwise}&\to& 0
        \end{array}.
\]
Indeed, consider $x\in\az$ with a defect $x_{01} = 00$. The phase of the $0$ in position $0$ is $0$ and the phase of the $0$ in position $1$ is $1$, so this corresponds to a particle $p_{01}$. Changing the position of the defect would not change the particle since the local phase would be modified accordingly.

The update function is defined in the intuitive manner: with $p_{01}$ evolving at speed $+1$ and $p_{10}$ at speed $-1$ and both particles being sent to $\emptyset$ in case of collision.
\[\forall x\in\az,\forall k\in\Z, \phi(x,k)=\left\{\begin{array}{cl}
\{k-1\}&\mbox{if }\pi(x)_k = p_{10} \mbox{ and }\pi(x)_{k-2}\neq p_{01}\\
\{k+1\}&\mbox{if }\pi(x)_k = p_{01} \mbox{ and }\pi(x)_{k+2}\neq p_{10}\\
\emptyset&\text{otherwise (and in particular if }\pi(x)_k=0)\end{array}\right.\]
We now check that the particle system satisfies all necessary conditions. To do that, one should verify that the update function is defined properly, that is:

\begin{align*}
\forall x\in\az,\ \forall k\in\Z,\  \pi(F(x))_{k-1}=p_{10}&\Leftrightarrow F(x)_{\{k-1,k\}} = 11\\
&\Leftrightarrow x_{[k-2,k+1]} \in \{1011, 0111, 1111\}\\
&\Leftrightarrow \pi(x)_k = p_{10}\mbox{ and }\pi(x)_{k+2}\neq p_{01},
\end{align*}
and similarly for $p_{01}$. This type of proof can become tedious due to the high number of cases but can be automated by straightforward enumeration, here of all patterns of length $4$. The different conditions follow from this property:
\begin{description}
\itemsep0em
 \item[Locality] Obvious by definition of $\phi$.
 \item[Redistribution] The claim can be restated as $\pi(F(x))_{k+1}=p_{10}\Leftrightarrow \pi(x)_k = p_{10}$ and $\phi(x,k)=\{k+1\}$, and similarly for $p_{01}$. The first condition follows. Since $\phi(x,k) = \emptyset$ when $k\notin\p(x)$ by definition of $\phi$, the second condition follows.
 \item[Disjunction] For $k<k'$, to have $\phi(x,k)>\phi(x,k')$, the only way would be to have $\pi(x)_{k'}=p_{01}$,
$\pi(x)_k=p_{10}$ and $k'=k+1$. In that case, by definition, $\phi(x,k)=\phi(x,k')=\emptyset$.
 \item[Coalescence and speeds] Obvious by definition of $\phi$.
 \end{description}

Therefore we can apply Corollary~\ref{cor:MainResult} and only one type of particle remains in $\Lambda_\mu(F_{184})$.\sk

Furthermore, since the collisions are of the form $p_{01} + p_{10} \to \emptyset$, it is clear that for all $x\in\az$, $\D_{p_{01}}(F_{184}(x)) - \D_{p_{01}}(x)=\D_{p_{10}}(F_{184}(x)) - \D_{p_{10}}(x)$. Therefore, which particle remains is decided according to whether $\mu([00])>\mu([11])$ or the opposite, both particles disappearing in case of equality. The third case follows from the fact that if $00, 11\notin\L(\Lambda_\mu(F_{184}))$, then $\Lambda_\mu(F_{184}) = \{\mathstrut^\infty 01^\infty, \mathstrut^\infty 10^\infty\}$ which support a unique measure $\meas{01}$.
\end{proof}

\subsubsection{$n$-state cyclic automaton}\label{section.cyclic}

The $n$-state cyclic automaton $C_n$ is a cellular automaton defined on the alphabet $\A = \Z/n\Z$ by the local rule \[c_n(x_{i-1}, x_i, x_{i+1}) = \left\{\begin{array}{ll}x_i+1&\mbox{if }x_{i-1} = x_i+1 \mbox{ or }x_{i+1}= x_i+1;\\x_i&\mbox{otherwise.}\end{array}\right.\]
See Figure~\ref{fig:particles} for an example of space-time diagram.\sk

This automaton was introduced by~\cite{Fisch}. In this paper, the author shows that for all Bernoulli measure $\mu$, the set $[i]_0$ (for $i\in\A$) is a $\mu$-attractor iff $n\geq5$: that is, $\mu(\{x\in\az\ :\ \exists T\in\N, \forall t\geq T, F^tx\in [i]\})>0$ for all $i$. Simulations starting from a random configuration suggest the following: for $n=3$ or $4$, monochromatic regions keep increasing in size; for $n\geq 5$, we observe the convergence to a fixed point where small regions are delimited by vertical lines. We use the main result to explain this observation.

\begin{proposition}
Define:
\begin{align*}
u_+ &= \{ab\in\A^2 : (b-a)\mod n = +1\};\\
u_- &= \{ab\in\A^2 : (b-a)\mod n = -1\};\\
u_0 &= \{ab\in\A^2 : (b-a)\mod n \notin\{-1,0,1\}\}.
\end{align*}
Then, for any measure $\mu\in\Merg((\Z/n\Z)^\Z)$, only one of those three sets may intersect the language of $\Lambda_\mu(C_n)$.\sk

If furthermore $\mu$ is a Bernoulli measure, then the persisting set can only be $u_0$.\end{proposition}

\begin{proof}
We consider the interface defects relatively to the decomposition $\Sigma = \bigsqcup_{i\in\A}\Sigma_i$, where $\Sigma_i= \{\mathstrut^{\infty}i^{\infty}\}$. $\Sigma$ is a $C_n$-invariant SFT of radius $r=2$, and defects are exactly
transitions between colours. Thus we define $\P = \{p_{ab}\ :\ ab\in\A^2, a\neq b\}$. One cell is enough to distinguish the domains ($\alpha=1$) and we obtain a morphism $\pi$ of radius $2$ defined by the local rule:
\[       \begin{array}{ccc}
         \A^2&\to&\P\cup\{0\}\\
         a\cdot a&\mapsto& 0\\
         a\cdot b&\mapsto& p_{ab}
         \end{array}\quad\quad\mbox{for all }a,b\in\A.
\]
Simulations suggest that $p_{ab}$ evolves at constant speed $+1$ if $ab\in u_+$, $-1$ if $ab\in u_-$ and $0$ if $ab\in u_0$. Particles progress at their assigned speed unless they meet another particle, in which case they interact according to the following chemistry: 
\begin{itemize}
\item $p_{ab}+p_{ba} \to \emptyset$ (if $p_{ab}$ has speed +1);
\item $p_{ab}+p_{bc}\to p_{ac}$ (if $p_{ab}$ and $p_{bc}$ have speeds $(+1,0)$ or $(0,-1)$, only when $n\geq 4$), or
\item $p_{ab}+p_{bc}+p_{cd} \to p_{ad}$. (if $p_{ab}, p_{bc}$ and $p_{cd}$ have speeds $+1, 0, -1$ respectively, which is only possible for $n=4$).
\end{itemize}

We group together the particles of same speed, writing $p_+ = \{p_{ab}\ :\ ab\in u_+ \}$ and $p_-$ and $p_0$ similarly. Formally, for $x\in\az$ and $k\in\Z$ the update function is defined as:
\[ \phi(x,k)=\left\{\begin{array}{cl}
\{k+1\}&\mbox{if }\pi(x)_k \in p_+\mbox{ and } \left\{\begin{array}{c}
\pi(x)_{k+1} \in p_+,\mbox{ or}\\
\pi(x)_{k+1}\notin \P \mbox{ and }\pi(x)_{k+2}\notin p_-
\end{array}\right.;\\
\{k-1\}&\mbox{if }\pi(x)_k \in p_-\mbox{ and } \left\{\begin{array}{c}
\pi(x)_{k-1} \in p_-,\mbox{ or}\\
\pi(x)_{k-1}\notin \P \mbox{ and }\pi(x)_{k-2}\notin p_+;
\end{array}\right.\\
\{k\}&\mbox{if }\pi(x)_k \in p_0\mbox{ and }\pi(x)_{k+1}\notin p_- \mbox{ and }\pi(x)_{k-1}\notin p_+\\
\emptyset&\text{otherwise (and in particular if }\pi(x)_k=0).\end{array}\right.\]
As previously, we can check that the update function actually describes the dynamics of the particles. For all $x\in\Z$, we check that:
\begin{align*}
\pi(F(x))_{1}\in p_+&\Leftrightarrow F(x)_{\{1,2\}} = ab \quad \mbox{with } a = b+1\\
&\Leftrightarrow x_{[0,3]} \in 
\left\{\begin{array}{ll}
abbc&\mbox{where } c\neq a\\
abc\_&\mbox{where } b=c+1\\
dacb&\mbox{where } d\neq a+1\mbox{ and } c=b-1 
\end{array}\right.\\
&\Leftrightarrow \left\{
\begin{array}{ll}
\pi(x)_0 \in p_+\mbox{ and }\pi(x)_{1}\notin \P \mbox{ and }\pi(x)_{2}\notin p_-, \mbox{ or }\\
\pi(x)_{1} \in p_0 \mbox{ and } \pi(x)_{2} \in p_- \mbox{ with good chemistry } (p_{a,a-2} + p_{a-2,a-1} \to p_{a,a-1})
\end{array}\right.
\end{align*}

and so on for other particle types, from which we deduce the hypotheses of Corollary~\ref{cor:MainResult}. Since $[p_+] = \pi([u_+])$ and so on, we obtain the result.\sk

\textbf{If $\mu$ is a Bernoulli measure:} Consider the ``mirror'' map $\gamma((a_k)_{k\in\Z}) = (a_{-k})_{k\in\Z}$. $\gamma$ is continuous, and thus measurable. We have $\mu(\gamma([u])) = \mu([u^{-1}]) = \mu([u])$, where $(u_1\cdots u_n)^{-1} = u_n\cdots u_1$. But $\pi(x)_k\in p_+\Leftrightarrow \pi(\gamma(x))_{-k}\in p_-$, and conversely; since $F\circ\gamma = \gamma\circ F$, all measures $\F^t\mu$ are $\gamma$-invariant, and thus no particle in $p_+$ or $p_-$ can persist in $\L(\pi(\Lambda_{\mu}(F)))$ (since otherwise, the symmetrical particle would persist too).
\end{proof}

For small values of $n$ or particular initial measures, this proposition can be refined in the following manner:
\begin{description}
\item[$\bm{n=3}$] $p_0$ is empty. Given the combinatorics of collisions, where a particle in $p_+$ can only
disappear by colliding with a particle in $p_-$, we see that particles in $p_+$ persist if and only if
$\pi_\ast\mu([p_+])>\pi_\ast\mu([p_-])$, and symmetrically. In the equality case (in particular, for any Bernoulli
measure), no defect can persist in the $\mu$-limit set, which means that $\Lambda_{\mu}(F)$ is a set of
monochromatic configurations.
\item[$\bm{n=4}$] If $\mu$ is a Bernoulli measure, the result of~\cite{Fisch} shows that $[i]_0$ cannot be a
$\mu$-attractor for any $i$. In other words, for $\mu$-almost all $x$, $F^t(x)$ does not converge, which means that particles in
$p_+$ or $p_-$ cross the central column infinitely often (even though their probability to appear tends to $0$).
This could not happen if particles in $p_0$ were persisting in $\pi(\Lambda_{\mu}(F))$, and thus $\Lambda_{\mu}(F)$ is a set of
monochromatic configurations. 
\item[$\bm{n\geq5}$] If $\mu$ is a nondegenerate Bernoulli measure, the result of~\cite{Fisch} shows that $[i]_0$ is a
$\mu$-attractor for all $i$. This means that some particles in $p_0$ persist in $\pi(\Lambda_{\mu}(F))$, and any
configuration of $\Lambda_{\mu}(F)$ contains only homogeneous regions separated by vertical lines. 
\end{description}

For $n=3$ or $4$, since $\Lambda_{\mu}(F)$ is a set of
monochromatic configurations we deduce that the sequence $(F^n\mu)_{n\in\N}$ converges to a convex combination of Dirac measures. However this method does not give any insight as to the coefficient of each component. As shown in~\cite{these-benjamin}, if $\mu$ is a Bernoulli measure then \[C_{3\ast}^t\mu \underset{t\to\infty}{\longrightarrow}\mu([2])\meas{0}+\mu([0])\meas{1}+\mu([1])\meas{2}.\]

The problem is open for the 4-cyclic cellular automaton.

\subsubsection{One-sided captive cellular automata}\label{section.captive}

We consider the family of captive cellular automata $F:\az\to\az$ of neighbourhood $\{0,1\}$, which means that the local rule $f:\A^{\{0,1\}}\to\A$ satisfies $f(a_0a_1)\in\{a_0,a_1\}$. See Figure~\ref{fig:particles} for an example of space-time diagram.

\begin{proposition}
 Let $F$ be a one-sided captive automaton and $\mu\in\Merg(\az)$.
 Define:
 \begin{align*}
  u_+ &= \{ab\in\A^2\ :\ a\neq b,\ f(a,b) = a\}\\
  u_- &= \{ab\in\A^2\ :\ a\neq b,\ f(a,b) = b\}
 \end{align*}
Then either $u_+\cap\L(\Lambda_\mu(F))=\emptyset$ or $u_-\cap\L(\Lambda_\mu(F))=\emptyset$.\sk

 If moreover, for all $a,b\in\mathcal{A}$, the local rule satisfies $f(ab)=f(ba)$ and $\mu$ is a Bernoulli measure, then $\Lambda_{\mu}(F)\subseteq \{\mathstrut^\infty a^\infty\ :\ a\in\A\}$ (no particle remains).
\end{proposition}

\begin{proof}
We consider the interface defects relative to the decomposition $\Sigma = \bigsqcup_{i\in\A}\Sigma_i$ where $\Sigma_i = \{\mathstrut^{\infty}i^{\infty}\}$ and obtain the same particles $\P$ and morphism $\pi$ as the $n$-state cyclic automata. $p_{ab}$ evolve at speed $-1$ if $f(a,b)=b$ and $0$ if $f(a,b)=a$, and we define $p_{-1}$ and $p_0$ accordingly. The update function is defined as follows:
\[\forall x\in\az,\ \forall k\in\Z,\ \phi(x,k)=
\left\{\begin{array}{cl}
       \{k\}&\text{if } \pi(x)_k\in p_0 \mbox{ and } \pi(x)_{k+1}\notin p_{-1}\\
       \{k-1\}&\text{if }\pi(x)_k\in p_{-1} \mbox{ and } \pi(x)_{k-1}\notin p_0\\
       \emptyset&\text{otherwise}\end{array}\right.\]
As in the two previous examples, we check by enumeration of cases that the update function describes the particle dynamics on all words of length 3:
\begin{align*}
\forall x\in\az,\ \forall k\in\Z,\ F(x)_k \in p_0 &\Leftrightarrow F(x)_{[k, k+1]} = ab \mbox{ where } a\neq b\mbox{ and }f(a,b) = a\\
&\Leftrightarrow x_{[k, k+1]} = abc \mbox{ where } b=c \mbox{ or }f(b,c) = b\\
&\Leftrightarrow \pi(x)_k\in p_0 \mbox{ and } \pi(x)_{k+1}\notin p_{-1}
\end{align*}

and deduce the properties of locality, redistribution, disjunction, coalescence and speed from there. The main result implies the theorem.

\textbf{If $\mu$ is a Bernoulli measure:} Then $\mu$ is invariant under the mirror map $\gamma$ and $F\circ\gamma = \gamma\circ F$ by hypothesis. As in the previous example, we conclude that no particle can persist in $\Lambda_\mu(F)$.
\end{proof}

\subsubsection{An automaton performing random walks}\label{section.randomwalk}

Let $F$ be defined on the alphabet $\A = (\Z/2\Z)^2$ on the neighbourhood $\{-2,\dots, 2\}$ by the local rule $f$ defined as follows:
\[f : (a_{-2},b_{-2}),\dots,(a_2,b_2) \mapsto (a_{-2}+a_2, c)\mbox{ where } c = \begin{array}{l}1 \mbox{ if }(a_{-1},b_{-1}) = (1,1)\mbox{ or }(a_0,b_0) = (0,1);\\0\mbox{ otherwise.}\end{array}\]
Intuitively, the first layer performs addition mod $2$ at distance 2, while the ones on the second layer behave as particles, moving right if the first layer contains a 1 and not moving if it contains a 0. Two colliding particles simply merge.

\begin{figure}[!ht]
 \begin{center}
  \includegraphics[width = \textwidth]{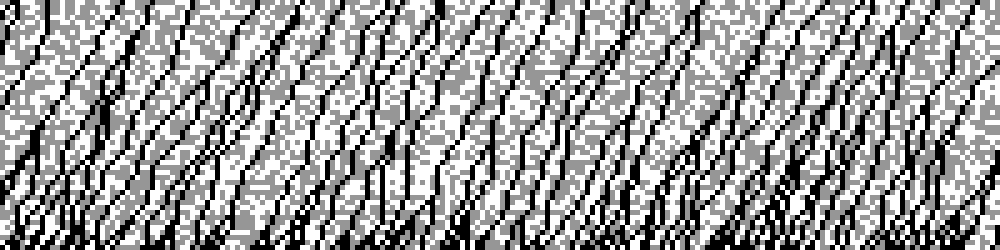}
 \end{center}
\caption{Automaton performing random walks iterated on the uniform measure. $\blacksquare$ is a particle, while the second layer is represented by $\square$ (0) or $\textcolor{gray}\blacksquare$ (1).}
\end{figure}

\begin{proposition}
Let $\nu\in\Merg((\Z/2\Z)^\Z)$ and $\mu = \lambda\times\nu$, where $\lambda$ is the uniform measure on $(\Z/2\Z)^\Z$. Then $\F^t\mu\underset{t\to\infty}\longrightarrow\lambda\times\meas{0}$.
\end{proposition}

\begin{proof}
Pivato's formalism is not necessary here. Consider the set of particles $\P = \{1\}$ and the morphism $\pi$ that is the projection on the second layer. The update function is defined as:
\[\forall x\in\az,\ \forall k\in\Z,\ \phi(x,k) = \left\{\begin{array}{cc}
               \{k+1\}&\mbox{if }x_k = (1,1);\\
               \{k\}&\mbox{if }x_k = (0,1);\\
               \emptyset&\mbox{otherwise.}
              \end{array}\right.
\]
Checking locality, redistribution, disjunction and coalescence is trivial here. Intuitively, each particle performs a random walk with independent steps and bias $\frac 12$. Thus Corollary~\ref{cor:MainResult} is not sufficient to conclude, and we need to use the general result of Theorem~\ref{prop:MainResult} by proving that $\{1\}$ clashes with itself. \sk

Writing $(a_k^t,b_k^t) = F^t(x)_k$, we have $a_k^t = \sum_{n=0}^t\binom tna_{k-2t+4n}^0\mod 2$ by straightforward induction. Now take some configuration $x$ with a particle at position $k$ and consider $\phi^t(x,k)$ the walk performed by the particle. First we prove that the particle performs a random walk as claimed above. We have:
\begin{align*}\phi^{t+1}(x,k)-\phi^t(x,k) &= a^t_{\phi^t(x,k)}\\
&= \left(\sum_{n=0}^t\binom nt a_{\phi^t(x,k)-2t+4n}^0\right)\mod 2\\
&= \left(a_{\phi^t(x,k)-2t}^0 + \sum_{n=1}^{t-1}\binom nt a_{\phi^t(x,k)-2t+4n}^0 + a_{\phi^t(x,k)+2t}^0 \right)\mod 2.\end{align*}
In the last line, we isolated the leftmost and rightmost term. Since $\phi^t(x,k)-2t$ is strictly decreasing and $\phi^t(x,k)+2t$ is strictly increasing in $t$, these terms do not appear in any $\phi^{t'+1}(x,k)-\phi^{t'}(x,k)$ for $t'<t$. Therefore, if $x$ is drawn according to a Bernoulli measure, the value of $a_{\phi^t(x,k)\pm 2t}^0$ is independent of the value of all $\phi^{t'+1}(x,k)-\phi^{t'}(x,k)$ for $t'<t$.

Formally, the behaviour of $\phi^{t'}(x,k)$ (for $t'\leq t$) only depends on the random variables $\{a_n^0\ :\ \phi^t(x,k)-2t+1\leq n\leq \phi^t(x,k)+2t-1\}$. Let $U$ be any event in the sigma-algebra generated by these variables. Then we have:
\begin{align*}
\mu\left(\phi^{t+1}(x,k) - \phi^t(x,k) = 0\ |\ U\right) =&\ \mu\left(a_{\phi^t(x,k)-2t}^0=0 \wedge \sum_{n=1}^{t}\binom nt a_{\phi^t(x,k)-2t+4n}^0 = 0\ |\ U\right)\\&+\mu\left(a_{\phi^t(x,k)-2t}^0=1 \wedge \sum_{n=1}^{t}\binom nt a_{\phi^t(x,k)-2t+4n}^0 = 1\ |\ U\right)\\
=&\ \mu\left(a_{\phi^t(x,k)-2t}^0=0\right)\cdot \mu\left(\sum_{n=1}^{t}\binom nt a_{\phi^t(x,k)-2t+4n}^0 = 0\ |\ U\right)\\
& + \mu\left(a_{\phi^t(x,k)-2t}^0=1\right)\cdot \mu\left(\sum_{n=1}^{t}\binom nt a_{\phi^t(x,k)-2t+4n}^0 = 1\ |\ U\right)\\
=&\ \frac 12,
\end{align*}
where the second step is by independence of the leftmost term from all the other variables, and the third step uses $\mu\left(a_{\phi^t(x,k)-2t}^0=0\right) = \frac 12$ since $\mu$ is the uniform Bernoulli measure on the first component. We proved that $(\phi^t(x,k))_{t\in\N}$ is a random walk with independent steps and bias $\frac 12$.\sk

To apply the theorem, we now prove that the particle clashes with itself. Note that the random walks performed by different particles are not independent; however, we prove that they are pairwise independent.

Let $k\in\N$. We prove that, when $x$ is chosen according to $\mu_k$ the conditional measure of $\mu$ relative to the event $\pi(x)_0=\pi(x)_k=1$, $\phi^t(x,k)-\phi^t(x,0)$ performs an unbiased and independent random walk with a ``death condition'' on 0 (particle collision).
Consider the evolution of $\phi^t(x,k)-\phi^t(x,0)$ at each step:
\begin{align*}
\delta_t(x) &= (\phi^{t+1}(x,k)-\phi^t(x,k))-(\phi^{t+1}(x,0)-\phi^{t}(x,0))\\
&= \underbrace{\left(\sum_{n=0}^t\binom nt a_{\phi^t(x,k)-2t+4n}^0\mod 2\right)}_{T1}-\underbrace{\left( \sum_{n=0}^t\binom nt
a_{\phi^t(x,0)-2t+4n}^0\mod 2\right)}_{T2}
\end{align*}

Note that the leftmost term of T1 is independent from T2 and all the past values of $\phi^{t+1}(x,k)-\phi^{t}(x,k)$; similarly, the rightmost term of T2 is independent from T1 and all past values of $\phi^{t+1}(x,0)-\phi^{t}(x,0)$. By the same argument as above, T1 and T2 are each worth $0$ or $1$ with probability $\frac 12$ independently of each other and of all values of $\delta_{t'}$ for $t'<t$. We conclude that $\delta_t$ takes values $-1, 0, +1$ with probability $\frac 14, \frac 12, \frac 14$ respectively independently of all values of $\delta_{t'}$ for $t'<t$.\sk

Therefore $\phi^t(x,k)-\phi^t(x,0)$ performs an unbiased and independent random walk. This implies that $\mu_k(\{x\ :\ \forall t, \phi^t(x,k)>\phi^t(x,0)\}) = 0$ (standard result in one-dimensional random walks). Since particles cannot cross, they almost surely end up being in interaction, and therefore $\{1\}$ clashes with itself $\mu$-almost surely. Applying the theorem, we find that no particle can remain in $\Lambda_\mu(F)$.

More precisely, if we write $\pi_i$ the morphism projecting on the $i$-th coordinate, $\pi_{2\ast}\F^t\mu\to\meas{0}$.
Since the addition mod 2 automaton is surjective, it leaves the uniform measure invariant.
Therefore $\pi_{1\ast}\F^t\mu = \lambda$, and we conclude that $\F^t\mu \to \lambda\times\meas{0}$.
\end{proof}

\subsection{Probabilistic cellular automata}\label{sec:probabilist}

\subsubsection{Adaptation of our formalism for probabilistic cellular automata}

This approach can be adapted to non-deterministic cellular automata, and in particular probabilistic cellular automata. We use here a generalised version of the standard definition.
\begin{definition}

 Let $\A$ be a finite alphabet and $\Nb\subset\Z$. We define a map that applies a bi-infinite sequence of local rules to a configuration componentwise:
 \[\Phi_\Nb:\begin{array}{rl}
         (\A^{\A^{\Nb}})^\Z\times\az &\to \az\\
         ((f_i)_{i\in\Z},(x_i)_{i\in\Z}) &\mapsto (f_i((x_{i+r})_{r\in\Nb})_{i\in\Z}.
        \end{array}
\]
\end{definition}

\begin{definition}[Generalised probabilistic cellular automata]

 A \define{generalised probabilistic cellular automaton} $\tilde F$ on the alphabet $\A$ with neighbourhood $\Nb$ is defined by a measure on bi-infinite sequence of local rules $\nu\in \Ms(\aarz)$.

For a configuration $x\in\az$, $\tilde F : \az\to\Msaz$ is then defined as:
\[\mbox{For any Borel set }U,\ \tilde F(x)(U)=\int_{\aarz} 1_U(\Phi_\Nb(f,x))\ud\nu(f).\]
\end{definition}

A deterministic cellular automaton $F$ defined by a local rule $f$ corresponds in this formalism to a Dirac $\nu = \meas{f}$ (in which case the image measure is a Dirac on the image configuration), and usual probabilistic cellular automata correspond to the case where $\nu$ is a Bernoulli measure; in other words, the local rule that applies at each coordinate is drawn independently among a finite set of local rules $\A^\Nb\to\A$.

\begin{definition}[Action on the space of measures]

A generalised probabilistic cellular automaton defined by a measure $\nu\in \Ms(\aarz)$ extends naturally to an action $\tilde\F:\Msaz\to\Msaz$ by defining \[\tilde\F\mu(U) = \int_{\az}\int_{\aarz} 1_U(\Phi_\Nb(f,x))\ud\nu(f)\ud\mu(x).\]
The $\mu$-limit measures set of $\tilde F$, $\V(\tilde F, \mu)$, is the set of cluster points of the sequence $(\tilde\F^t\mu)_{t\in\N}$, and the $\mu$-limit set can be defined as \[\Lambda_\mu(\tilde F) = \overline{\bigcup_{\eta\in\V(\tilde F, \mu)}\supp\ \eta}.\] 
\end{definition}

The definitions of a particle system extend directly, except that the update function also depends on the choice of the local rules as well as on the configuration. Therefore we write $\phi(x,n,(f_i))$ instead of $\phi(x,n)$, where $x\in\az, n\in\Z$ and $(f_i)\in\aarz$, and the composition notation is simplified as follows (inductively):
\[\phi^t\left(x,n,(f^k)_{0\leq k< t}\right) = \bigcup_{m\in\phi(x,n,f^1)} \phi^{t-1}\left(\Phi_\Nb(f^1, x), m, (f^k)_{1\leq k<t}\right),\]
where each $f^t\in\aarz$ is a bi-infinite sequence of local rules.

A particle system is said to be coalescent $\nu$-almost surely if the coalescence conditions hold for all $x\in\az$ and $\nu$-almost every
$f\in\aarz$, and a particle $p\in\P$ has speed $v$ $\nu^\infty$-almost surely if the speed conditions hold for $\nu^\infty$-almost every sequence 
$(f^t)_{t\in\N}$, where $\nu^\infty$ is the product measure (i.e. each $f^t$ is drawn independently according to $\nu$).
The clashing conditions are extended similarly. \sk

\begin{theorem}[Qualitative result for probabilistic automata]\label{prop:MainResultProb}

Let $\tilde F:\az\to\Msaz$ be a probabilistic cellular automaton defined by $\nu\in\Ms(\aarz)$, $\mu$ an initial $\s$-ergodic measure and $(\P,\pi,\phi)$ a $\nu^\infty$-almost surely coalescent particle system for $\tilde F$ where $\P$ can be partitioned into sets $\P_1\dots \P_n$ such that, for any $i<j$, $\P_i$ clashes with $\P_j$ $\mu,\nu^\infty$-almost surely.

Then all particles appearing in the $\mu$-limit set belong to the same subset, i.e. there exists an $i$ such that
\[\forall p\in\P,\ p\in\L(\pi(\Lambda_\mu(F)))\Rightarrow p\in\P_i.\]
If furthermore there exists a $j$ such that $\P_j$ clashes with itself $\mu,\nu^\infty$-almost surely, then this set of particles does not appear in
the $\mu$-limit set, i.e. \[\forall p\in\P,\ p\in\L(\pi(\Lambda_\mu(F)))\Rightarrow p\notin\P_j.\]
\end{theorem}

\begin{corollary}[Main result with speedy particles - probabilistic automata]

Let $\tilde F:\az\to\az$ be a probabilistic cellular automaton defined by $\nu\in\Ms(\aarz)$, $\mu$ an initial $\s$-ergodic measure and $(\P,\pi,\phi)$ a $\nu^\infty$-almost surely coalescent particle system for $\tilde F$. Assume that each particle $p\in\P$ has speed $v_p\in\R$ $\nu^\infty$-almost
surely, then there is a speed  $v\in\R$ such that:
\[\forall p\in\P,\ p\in\L(\pi(\Lambda_\mu(F)))\Rightarrow v_p=v.\]
\end{corollary}

The proof of these statements are exactly the same as the proofs of Theorem~\ref{prop:MainResult} and Corollary~\ref{cor:MainResult}, except that every statement in the proof holds $\nu^\infty$-almost surely.

This Theorem and Corollary can be applied for different probabilistic cellular automata, for example when we mix two one sided captive CA (see Figure~\ref{fig:PCACaptive}). We are going to detail two examples from the literature and obtain new information about its limit measures (Section~\ref{section.FatesDensity} and~\ref{section.Regnault-Remilia-line}). 

\begin{figure}[ht]
\begin{center}
\includegraphics[width=9cm]{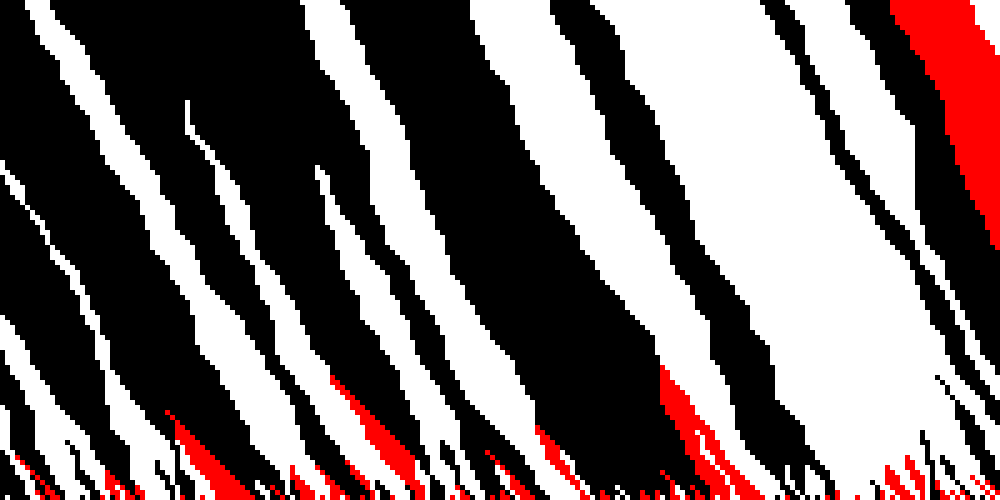}
\end{center}
 \caption{Example of probabilistic cellular automata where the update of each cell is chosen between two one sided captive CA.}
 \label{fig:PCACaptive}
\end{figure}

\subsubsection{Example: Majority-traffic PCA}\label{section.FatesDensity}

For any real $p\in[0,1]$, consider the probabilistic automaton $\tilde F$ on the alphabet $\{0,1\}$ defined on the neighbourhood $\Nb = \{-1,0,1\}$ by local rules drawn independently between the traffic rule (rule \#184 defined in Section~\ref{section.184}) with probability $p$ and the majority rule (rule \#232 where $F(x)_i=1 $ if and only if $x_{-1}+x_0+x_1\geq 2$) with  probability $1-p$. This corresponds to the case where $\nu$ is a Bernoulli measure.

This automaton was introduced by Fatès in \cite{Fates} as a candidate to solve the density classification problem. 

\begin{figure}[!ht]
\begin{center}
 \begin{tabular}{c|c|c}
  \includegraphics[width = .3\textwidth]{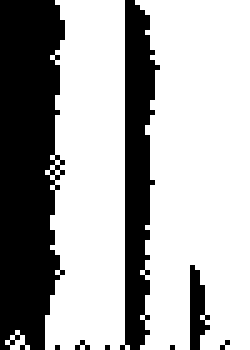}&
  \includegraphics[width = .3\textwidth]{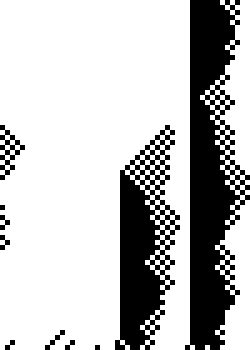}&
  \includegraphics[width = .3\textwidth]{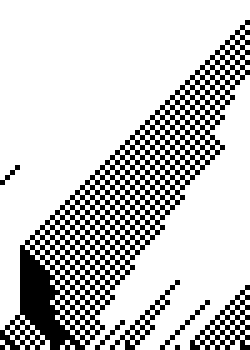}\\
  $p = \frac 14$&$p = \frac 12$&$p = \frac 34$
 \end{tabular}
\end{center}
\caption{Dynamics of the traffic-majority automaton iterated on the initial measure $\Ber(\frac 35, \frac 25)$. 
Density classification is more efficient with $p$ close to 1.}
\end{figure}

In \cite{Busic2013}, the authors completely describe the invariant measures of this PCA. In the continuity of the rest of the article, we are interested in the convergence properties of all $\s$-ergodic initial measure. None of these two results imply the other.

\begin{proposition}[Prop. 5.5 in \cite{Busic2013}]
For any $p$ in $[0,1]$, the set of $\tilde F$-invariant measures is the set of convex combinations of $\meas{0}, \meas{1}$ and $\meas{01}$.
\end{proposition}

\begin{proposition}
Let $\mu\in\Merg(\az)$ and $p$ be a real in $[0,1]$. Then \[\Lambda_\mu(\tilde{F}) \subset \{\per{0}, \per{1}, \per{(01)}, \per{(10)}\}.\]
As a consequence, any $\mu$-limit measure of $(\tilde{F}^t_\ast\mu)_{t\in\N}$ is a convex combination of $\meas{0}, \meas{1}$ and $\meas{01}$.
\end{proposition}

\begin{proof}
The cases $p=0,1$ correspond to deterministic automata and can be treated easily.\sk

The visual intuition suggests to consider interface defects according to the decomposition $\Sigma_0\sqcup\Sigma_1\sqcup\Sigma_2$, where
$\Sigma_0 = \{\per{0}\}$, $\Sigma_1 = \{\per{1}\}$ (monochromatic subshifts) and $\Sigma_2 = \{ \per{(01)}, \per{(10)}\}$ (chequerboard subshift),
since those SFTs are invariant under the action of both rules. The set of particles would be $\P = \{p_{i,j} : i\neq j\in\{0,1,2\}\}$.

However, as Figure~\ref{fig:trafmaj} shows, the particle $p_{10}$ can ``explode'' and give birth to two particles $p_{12}$ and
$p_{20}$, contradicting the condition of coalescence. To solve this problem, we tweak the particle system by replacing
each particle $p_{10}$ by one particle $p_{12}$ and one particle $p_{20}$.

\begin{center}
\captionsetup{type=figure}
\begin{tabular}{cl}
\hspace{-5.8cm}
\includegraphics[width=\textwidth, trim = 0 -25 0 0]{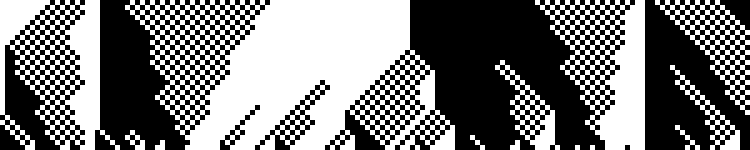}&
\hspace{-12.5cm}
 \begin{tikzpicture}
  \draw (1.4, -.8) node {\footnotesize{``explosion''}};
  \draw[->, thick] (.8,-1) -- (0,-1.65);
  \draw (-.4,-2.2) node {\footnotesize{$p_{12}$}};
  \draw[->, thick] (-.4,-2.1) -- (-.15,-1.95);
  \draw (.2,-2.2) node {\footnotesize{$p_{20}$}};
  \draw[->, thick] (.2,-2.1) -- (-0.05,-1.95);
  \draw (2.2,0.2) node {\footnotesize{$p_{02}+p_{20} = \emptyset$}};
  \draw[->, thick] (2.2,0) -- (2.65,-.5);
  \draw (3.3,.8) node {\footnotesize{$p_{02}+p_{21} = p_{01}$}};
  \draw[->, thick] (3.5,.6) -- (4.35, .15);
  \draw (5.3,-2.1) node {\footnotesize{$p_{01}$}};
 \end{tikzpicture}
\end{tabular}
 \caption{Fatès' traffic-majority probabilistic automaton, with $p=\frac 34$.}
 \label{fig:trafmaj}
 \end{center}
 
The corresponding morphism $\pi$ is defined on the neighbourhood $\{0,\dots,3\}$ by the local rule:
\[\begin{array}{cccccc}
   0011&\mapsto&p_{01}&
\quad110\_&\mapsto&p_{12}\\
   0010&\mapsto&p_{02}&
\quad  100\_&\mapsto&p_{20}\\
1011&\mapsto&p_{21}&
  \quad \mbox{otherwise}&\mapsto&0
   \end{array}
\]
where the wildcards $\_$ can take both values.

Empirically, the particle behaviour without interactions is as follows. Regardless of the rule that is applied, $p_{01},p_{02}$ and $p_{21}$ move at a constant speed $0$, $+1$ and $-1$ respectively. A particle $p_{12}$ moves at speed $-1$ if rule $\#184$ is applied at its position and at speed $+1$ otherwise (independent random walk with bias $1-2p$), except if a particle $p_{20}$ prevents its movement to the right, in which case it does not move. The particle $p_{20}$ behaves symmetrically. As an abuse of notation, we denote for easier reading $\pi(x)_0 = p_{12}^-$ if $\pi(x)_0 = p_{12}$ and $f_k = \#184$ and so on.

Particle interactions are of the form $p_{ij}+p_{ji}\to\emptyset$, $p_{ij}+p_{jk}+p_{ki}\to\emptyset$, or $p_{ij}+p_{jk}\to p_{ik}$, although some of these can not happen. Interactions involve particles at distance at most 3.

Formally, we prove through exhaustive case enumeration of all patterns of length 7 and possible local rules that:
\begin{align*}
\pi(F(x))_0 = p_{01}\ \Longleftrightarrow &\ (\pi(x)_0 = p_{01} \band \pi(x)_{-1} \neq p_{20}^+ \band \pi(x)_2 \neq p_{12}^-)
&p_{01} \textrm{ moves at speed $0$}\\
				     & \bor (\pi(x)_{-1} = p_{02} \band \pi(x)_{1} = p_{21})
& p_{02} + p_{21} \to p_{01}\\
\pi(F(x))_0 = p_{02}\ \Longleftrightarrow &\ (\pi(x)_{-1} = p_{02} \band \pi(x)_1 \notin \{p_{21}, p_{20}^-\})
&p_{02} \textrm{ moves at speed $+1$}\\
				     & \bor (\pi(x)_0 = p_{01} \band \pi(x)_{2} = p_{12}^- \band \pi(x)_{-1} \neq p_{20}^+)
& p_{01} + p_{12} \to p_{02}\\
\pi(F(x))_0 = p_{12}\ \Longleftrightarrow &\ (\pi(x)_{-1} = p_{12}^+ \band \pi(x)_0 \notin \{p_{21}, p_{20}^-\})
&p_{12}^+ \textrm{ moves at speed $+1$}\\
				        & \bor (\pi(x)_0 = p_{12}^+ \band \pi(x)_{1} = p_{20}^-)
&p_{12}^+ \textrm{ is blocked}\\
					& \bor (\pi(x)_1 = p_{12}^- \band \pi(x)_{-1} \neq p_{01})
&p_{12}^- \textrm{ moves at speed $-1$}\\
\pi(F(x))_0 = p_{20}\ \Longleftrightarrow &\ (\pi(x)_{-1} = p_{20}^+ \band \pi(x)_1 \neq p_{01})
&p_{20}^+ \textrm{ moves at speed $+1$}\\
				     & \bor (\pi(x)_0 = p_{20}^- \band \pi(x)_{-1} = p_{12}^+)
&p_{20}^- \textrm{ is blocked}\\
				     & \bor (\pi(x)_{1} = p_{20^-} \band \pi(x)_{-1} \neq p_{02} \band \pi(x)_0 = p_{12}^+)
&p_{20}^- \textrm{ moves at speed $-1$}\\
\pi(F(x))_0 = p_{21}\ \Longleftrightarrow &\ (\pi(x)_{1} = p_{21} \band \pi(x)_{-1} \neq p_{02} \band \pi(x)_0 \neq p_{12}^+)
&p_{21} \textrm{ moves at speed $-1$}\\
				     & \bor (\pi(x)_{-1} = p_{20}^+ \band \pi(x)_0 = p_{01} \band \pi(x)_2 \neq p_{12}^-)
& p_{20} + p_{01} \to p_{21}\\
\pi(F(x))_0 = 0\qquad& \textrm{in all other cases (including other possible interactions)}
\end{align*}

Using this statement, it is straightforward though tedious to define formally the update function, and the various conditions of locality, disjunction, particle control, surjectivity and coalescence are proved similarly to the previous examples.\sk

Assume $p\geq \frac12$. We show that no particle can remain asymptotically by applying the main result on the sets $(\P_i)_{0\leq i\leq 4}$: $\{p_{02}\}$, $\{p_{20}\}$, $\{p_{01}\}$, $\{p_{12}\}$ and $\{p_{21}\}$. We need only to show the clashes relative to the second and fourth sets since all other clashes are consequences of the speed of these particles.

Let $k\in\N$ and $x$ be such that $\pi(x)_0 = p_{02}$ and $\pi(x)_k \in \{p_{12}, p_{20}\}$. Since $p_{02}$ progresses at speed 1, the distance $\phi^t(x,k)-\phi^t(x,0)$ cannot increase, and it decreases by at least one with probability $p$ (respectively $1-p$). It is clear that the particles end up in interaction $\nu^\infty$-almost surely. Showing that $p_{12}$ and $p_{20}$ clash with $p_{21}$ is symmetric.

Let $x$ be such that $\pi(x)_0 = p_{20}$ and $\pi(x)_k = p_{01}$. As long as there are no interactions, the distance $\phi^t(x,k)-\phi^t(x,0) = -\phi^t(x,0)$ performs an independent random walk of bias $2p-1$, where a increasing step is sometimes replaced by a constant step. Such a random walk reaches $0$ $\nu^\infty$-almost surely, which shows that the particles end up in interaction. The clashes between $p_{01}$ and $p_{12}$, and between $\{p_{20}\}$ and $\{p_{12}\}$, are proved in a similar manner. The same proof holds for $p\leq \frac12$ by exchanging the roles of $p_{20}$ and $p_{12}$.\sk

Applying Theorem~\ref{prop:MainResultProb}, we conclude that only one particle $p_{ij}$ can remain in the $\mu$-limit set. This result can be improved further: consider $V_k = \{ x\in\Lambda_\mu(F) : \pi(x)_k = p_{ij}\}$. Configurations in $V_k$ are of the form $y\cdot z$, where $y\in\A^{]-\infty, k]}$ is admissible for $\Sigma_i$ and $z\in\A^{[k+1, +\infty[}$ is admissible for $\Sigma_j$; in particular, they contain only one particle. For any measure $\eta\in\V(\tilde F,\mu)$, $\eta(V_k)$ is independent from $k$ by $\s$-invariance, and $\eta\left(\bigcup_kV_k\right) = \sum_k\eta(V_k)\leq 1$ by disjunction of the $(V_k)_{k\in\Z}$. Consequently, $\eta(V_k)=0$, which means $V_k \notin\supp(\eta)$. We conclude that no particle remain in the $\mu$-limit set, or in other words, $\Lambda_\mu(F) \subset \Sigma_0\cup\Sigma_1\cup\Sigma_2$.
\end{proof}

\subsubsection{Example: Approximation of a line}\label{section.Regnault-Remilia-line}

A finite word of $\{0,1\}^\ast$ can be seen as a finite curve in $\Z^2$ taking its origin in $(0,0)$, moving right on a $0$ and up on a $1$. In~\cite{Regnault-Remilia-2015}, the authors introduce for any $\alpha\in\mathbb{Q}\cap[0,1]$ a random process that, starting from a finite word $w\in \{0,1\}^\ast$ whose frequency of symbols $1$ is $\alpha$, organises bits through local flips to obtain asymptotically a discrete segment of slope $\alpha$. 

We adapt these processes so that the flips are performed in parallel and on an infinite configuration, which gives a probabilistic cellular automaton for every slope $\alpha\in\mathbb{Q}\cap[0,1]$. We consider the action of this PCA on any initial $\s$-ergodic measure satisfying $\mu([1])=\alpha$. Using Theorem~\ref{prop:MainResultProb}, we show that the sequence of measures converges towards the measure supported by a periodic configuration representing a discrete line of slope $\alpha$. To simplify the presentation, we consider here that $\alpha=\frac{1}{2}$; the method can be easily generalised to other slopes.\bigskip

Define the following local rules:
\begin{itemize}
\item $i$ is the identity;
\item $r(x_{-2}, x_{-1},x_0,x_1)=\begin{cases}x_0&\text{ if }x_{-2}x_{-1}x_0x_1=0101\textrm{ or }1010, \\
x_{-1}&\text{ otherwise};\end{cases} $
\item $\ell(x_{-1},x_0,x_1,x_2)=\begin{cases}x_{0}&\text{ if }x_{-1}x_0x_1x_2=0101\textrm{ or }1010,\\x_1&\text{ otherwise}.\end{cases} $
\end{itemize}

Let $\tilde F_{\textrm{line}}$ be a probabilistic cellular automaton (represented in Figure~\ref{fig:line}) defined by a $\s$-ergodic measure $\nu\in\mathcal{M}_{\s}(\{g_0,g_1,g_{-1}\}^{\Z})$ whose support is the subshift of finite type defined by the set of forbidden patterns \[\{ir,\ \ell\ell,\ \ell i,\ rr,\ r\ell\}.\] To put it more simply, any time the local rules in two consecutive cells are $\ell$ and $r$ (which happens with positive probability), the probabilistic CA permutes these two letters, except if they are at the centre of a four-letter words $1010$ or $0101$. In any other situation, it acts as the identity.

\begin{proposition}
Let $\mu\in\Merg(\az)$. Then:
\begin{align*}
 \mu([00])>\mu([11])&\Rightarrow 11\notin \Lambda_\mu(\tilde F_{\textrm{line}});\\
 \mu([00])<\mu([11])&\Rightarrow 00\notin \Lambda_\mu(\tilde F_{\textrm{line}});\\
 \mu([00])=\mu([11])&\Rightarrow \tilde F_{\textrm{line}}^t\mu \to \meas{01}.
\end{align*}
\end{proposition}

\begin{proof}We consider the dislocation defects with regards to the chequerboard SFT $\Sigma = \{\per{(01)},\per{(10)}\}$. As in Section~\ref{section.184}, we obtain the particles $11\mapsto p_{01}$ and $00\mapsto p_{10}$. A particle $p_{10}$ at position $0$ moves at speed $+2$ if $\ell$ is applied at position $1$, at speed $-2$ if $\ell$ is applied on position $-1$, and at speed $0$ otherwise. The particle $p_{01}$ is symmetrical and they annihilate on contact. Indeed, we check by straightforward case enumeration that:

\begin{align*}
\pi(\tilde F_{\textrm{line}}(x))_{0} = p_{10}&\Leftrightarrow \tilde F_{\textrm{line}}(x)_{[0,1]} = 00\\
&\Leftrightarrow \left\{\begin{array}{ll}
x_{[0,1]} = 00 &\mbox{ and } f_{-1},\ x_{-1}\neq (\ell,1),\ (f_1,x_2)\neq (\ell,1) \\
x_{[-2,-1]} = 00 &\mbox{ and } f_{-1} = \ell,\ x_1 = 0\\
x_{[2,3]} = 00 &\mbox{ and } f_{1} = \ell,\ x_0 = 0
\end{array}\right.\\
&\Leftrightarrow \left\{\begin{array}{ll}
\pi(x)_{0} = p_{10} &\mbox{ at speed } 0 \\
\pi(x)_{-2} = p_{10} &\mbox{ at speed } +2 \mbox{ and } \pi(x)_{0} \neq p_{10}\\
\pi(x)_{2} = p_{10} &\mbox{ at speed } -2 \mbox{ and } \pi(x)_{0} \neq p_{10}
\end{array}\right.,
\end{align*}
and similarly for $p_{01}$. From this we deduce the various hypotheses of theorem, including $\nu^\infty$-almost sure clashing which stems from the fact that particles perform random walks. The exact statement of the proposition follows through the same arguments as in Section~\ref{section.184}; in particular, the third case corresponds to the discrete line of slope $\frac{1}{2}$.\end{proof}

\begin{figure}[h!]
\begin{center}
\includegraphics[width=9cm]{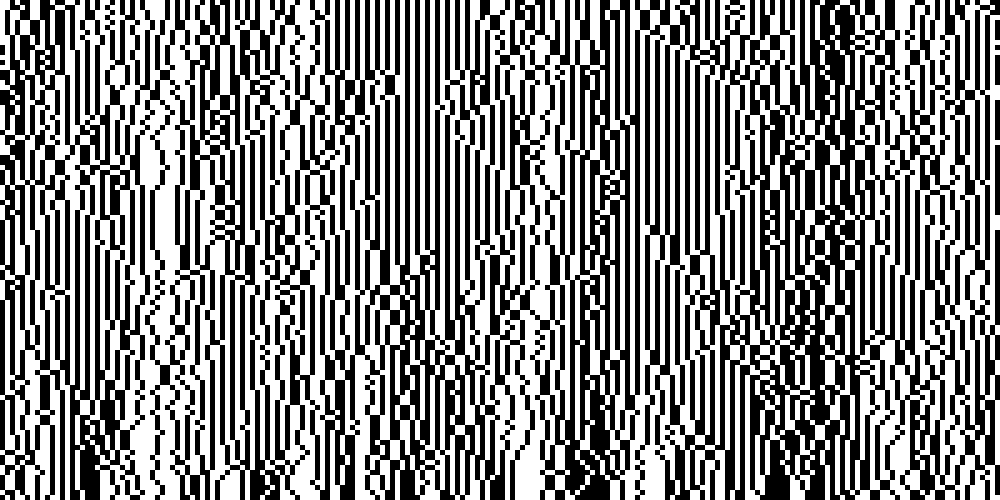}
\end{center}
 \caption{Example of a space-time diagram of $\tilde F_{\textrm{line}}$, where $\nu$ is the Markov measure maximising the entropy of the subshift of finite type defined by the forbidden patterns $\{f_0f_{-1},\ f_1f_1,\ f_1f_0,\ f_{-1}f_{-1},\ f_{-1}f_{1}\}$.}
 \label{fig:line}
\end{figure}

%%%%
%%%%%%%%%%%%%%%%%%%%%%%%%%%%%%%%%%%%%%%%%%%
%%%%%%%%%%%%%%%%%%%%%%%%%%%%%%%%%%%%%%%%%%%
%%%%

\newcommand\Mix{\mathcal Mix}
\newcommand\G{G_\ast}

\section{Particle-based organisation: quantitative results}\label{sec:brown}

For some cellular automata with simple defect dynamics, the previous results can be refined with a quantitative approach: that is, to determine the asymptotic distribution of random variables related to the particles. In \cite{EntryTimes}, P. K\r urka, E. Formenti and A. Dennunzio considered $T_n(x)$, the entry time after time $n$ on  the initial configuration $x$, which is the waiting time before a particle appears in a given position after time $n$. They restricted their study to a gliders automaton, which is a cellular automaton on 3 states: a background state and two particles evolving at speeds 0 and -1 that annihilate on contact. Thus, we have one entry time for each type of particle ($T_n^+(x)$ and $T_n^-(x)$). When the initial configuration is drawn according to the Bernoulli measure of parameters $(\frac 12, 0, \frac 12)$, which means that each cell contains, independently, a particle of each type with probability $\frac 12$, they proved that: \[\forall \alpha\in\R^+,~\mu\left(\frac {T_n^-(x)}n
\leq \alpha\right)\underset{n\to\infty}{\longrightarrow}\frac 2\pi \arctan\sqrt \alpha.\]
They also suggested to develop formal tools in order to be able to handle more complex automata, starting with the $(-1,1)$ symmetric case.

In Section~\ref{sec:entry}, we extend this result to allow arbitrary values for the particle speeds $v_-$ and $v_+$, and relax the conditions on the initial measure to some mixing conditions. Then, when $v_-< 0$ and $v_+\geq 0$, we have:
\[\forall \alpha\in\R^+,~\mu\left(\frac{T_n^-(x)}n\leq \alpha\right) \underset{n\to\infty}{\longrightarrow} \frac
2\pi\arctan\left(\sqrt{\frac{-v_-\alpha}{v_+-v_-+v_+\alpha}}\right),\] and symmetrically if we exchange $+$ and $-$. The proof relies on the fact that the behaviour of gliders automata can be characterised by some random walk process; this idea was introduced by V. Belitsky and P. Ferrari in \cite{BelitskyFerrari-1995} and was already used in \cite{Kurka-Maass-2000} and \cite{EntryTimes}. In our case, a particle appearing in a position corresponds to a minimum between two concurrent random walks. The new tool here is that under $\alpha$-mixing conditions, we rescale this process and approximate it with a Brownian motion. Thus we obtain the explicit asymptotic distribution of entry times.

This method, consisting in associating a random walk to each gliders automata and studying this random walk using scale invariance, is not limited to this particular conjecture concerning entry times. Indeed, we see in the next two sections that it can be used to study the asymptotic behaviour of two other, arguably more natural, parameters: the particle density at time $t$ and the rate of convergence to the limit measure. However, we obtain only an upper bound instead of an explicit asymptotic distribution. There is no doubt this method can be adapted to other parameters in a similar way.

Furthermore, these results can be extended to other automata with similar behaviour, such as those in Figure~\ref{fig:particles}, by factorising them onto a gliders automaton. This point is discussed in Section~\ref{sec:extensions}. This method is more difficult to generalise when there is birth of particle, even in a simple case such as the $4$-cyclic cellular automaton.

\subsection{Gliders automata and random walks}\label{sec:walks}

In this section we give the definition and the first properties of the class of gliders automata.

\begin{definition}[Gliders automata]

 Let $v_-<v_+ \in \Z$. The \define{$(v_-,v_+)$-gliders automaton} (or GA) $G$ is the cellular automaton of neighbourhood
$[-v_+, -v_-]$ defined on the alphabet $\A = \{-1, 0, +1\}$ by the local rule:
\[f(x_{-v_+}\dots x_{-v_-}) = \left\{\begin{array}{cl}+1&\tx{if }x_{-v_+}=+1\tx{ and }\forall N\leq -v_-,\sum_{n=-v_++1}^{N}
x_n\geq 0\\-1&\tx{if }x_{-v_-}=-1\tx{ and }\forall N\geq -v_+,\sum_{n=N}^{-v_--1} x_n\leq
0\\0&\tx{otherwise.}\end{array}\right.\]
\nomentry{$(v_-,v_+)$-GA}{Gliders automaton with particle speeds $v_-,v_+$}
\end{definition}

In all the following, $\A = \{-1,0,+1\}$ and the diagrams are represented with the convention $\square = 0, \blacksquare
= +1, \textcolor{red}{\blacksquare} = -1$.
\begin{figure}[!ht]
\begin{center}
 \includegraphics[width = \textwidth, trim = 0 0 0 300, clip]{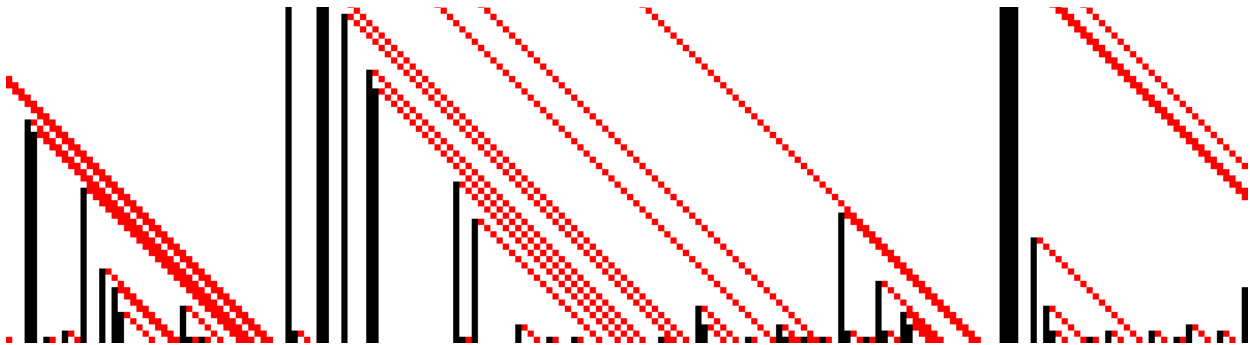}
 \caption{Space-time diagram of the $(-1,0)$-gliders automaton on a random initial configuration.}
\end{center}
\end{figure}

Our results apply on automata with simple defects dynamics, namely, automata admitting a particle system with $\P = \{\pm1\}$ and whose update function corresponds to a gliders automaton. We first prove our results for gliders automata before generalising them in Section~\ref{sec:extensions}. Let us introduce some tools that turn the study of the dynamics of a gliders automaton into the study of some random walk.

\begin{definition}[Random walk associated with a configuration]\label{def:walk}

Let $x\in\{-1,0,1\}^\Z$. Define the partial sums $S_x$ by: 
\[S_x(0) = 0\quad \mbox{and}\quad \forall k\in\Z,\ S_x(k+1)-S_x(k) = x_k.\] 

We extend $S_x$ to $\R$ by piecewise linear interpolation: $S_x(t) = (\lceil t\rceil-t)S_x(\lfloor t\rfloor) + (t-\lfloor t\rfloor)S_x(\lceil t\rceil)$ for $t\in\R\backslash \Z$. We also introduce the rescaled process $S_x^k:t\mapsto\frac{S_x(kt)}{\sqrt k}$.
\nomentry{$S_x(t)$}{Random walk associated with the configuration x}
\nomentry{$S^k_x(t)$}{Rescaled process corresponding to $S_x$}

\end{definition}

This random walk is simpler to study than the space-time diagram of the gliders automaton, and actually contains the
same amount of information, as shown by the following technical lemmas.

\begin{definition}
Let $f:\R\to\R$ and $U\subset\R$. We define $\argmin{U} f$ by:
 \[\forall t\in U,\ t= \argmin{U} f \Longleftrightarrow \forall t'\in U\backslash\{t\}, f(t) < f(t').\]
 
 In other words, $t$ realises the strict minimum of $f$ on $U$; this point is not always defined.
\end{definition}

\begin{lemma}\label{lem:MarAlea}
 Let $G$ be the $(v_-,v_+)$-gliders automaton. For all $j\in\Z$ and $n\geq1$,
\begin{align*}
 j = \argmin{[j,\ j+n]}S_{G(x)} \Longleftrightarrow j-v_+ = \argmin{[j-v_+,\ j+n-v_-]}{S_x},\\
 j = \argmin{[j-n,\ j]} S_{G(x)} \Longleftrightarrow j-v_- = \argmin{[j-n-v_+,\ j-v_-]} S_x.
\end{align*}
\end{lemma}

\begin{proof}
 We prove those equivalences by induction on $n$. At each step, we prove only the first
equivalence, the other one being symmetric.

\begin{description}
 \item[Base case]
\begin{align*}
S_{G(x)}(j) < S_{G(x)}(j+1) &\Leftrightarrow G(x)_j = +1\\
                            &\Leftrightarrow x_{j-v_+} = +1 \tx{ and } \forall N\leq -v_-,\sum_{t=-v_++1}^N x_{j+t}\geq
0\\
                            &\Leftrightarrow  S_{x}(j-v_+) < \underset{[j+1-v_+,\ j+1-v_-]}{\min} S_x.
\end{align*}
\item[Induction] Assume both equivalences hold for some $n\geq 1$.

Suppose $j = \argmin{[j,\ j+n+1]}S_{G(x)}$. In particular $j = \argmin{[j,\ j+n]}S_{G(x)}$, and by
induction hypothesis $j-v_+ = \argmin{[j-v_+,\ j+n-v_-]}{S_x}$.
We distinguish two cases:

\begin{itemize}
\item if $S_x(j+n-v_-+1) > S_x(j-v_+)$, then $j-v_+ = \argmin{[j-v_+,\ j+n-v_-+1]}{S_x}$ and we conclude;
\item otherwise, this means that $S_x(j+n-v_-+1) = S_x(j-v_+)$ (the walk can decrease by at most one at each step),
and thus \[j+n-v_-+1 = \argmin{[j-v_++1,\ j+n-v_-+1]}{S_x}.\]
By induction hypothesis, \[j+n+1 = \argmin{[j+1,\ j+n+1]}{S_{G(x)}},\]
and in particular $S_{G(x)}(j+n+1)<S_{G(x)}(j+1)$. Therefore $S_{G(x)}(j+n+1)\leq S_{G(x)}(j)$, a
contradiction with the first assumption.
\end{itemize}
The converse is proved in a similar manner.
\end{description}
\end{proof}

\begin{lemma}\label{lem:Min} Let $G$ be the $(v_-,v_+)$-gliders automaton. For all $j\in\Z$ and $k\geq 0,$
\begin{align*}G^t(x)_j= -1&\Longleftrightarrow j-v_-t+1 = \argmin{[j-v_+t,\ j-v_-t+1]}S_x\\
  G^t(x)_j= +1&\Longleftrightarrow j-v_+t = \argmin{[j-v_+t,\ j-v_-t+1]} S_x.
\end{align*}
This is illustrated in Figure~\ref{fig:LemMin}.
\end{lemma}

\begin{figure}[!ht]
\begin{center}
\begin{tabular}{cc}
\includegraphics{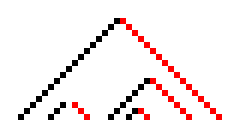}&\vspace{-4.155cm}\\
\begin{tikzpicture}
\foreach \x/\y/\z in {-2/0/0, -1/0/0, 0/0/1, 1/1/1, 2/1/1, 3/1/1, 4/1/1, 5/1/2, 6/2/2, 7/2/2, 8/2/2, 9/2/2, 10/2/2,
11/2/1, 12/1/1, 13/1/1, 14/1/1, 15/1/2, 16/2/2, 17/2/2, 18/2/3, 19/3/3, 20/3/3,  21/3/2, 22/2/2, 23/2/2, 24/2/2, 25/2/2,
26/2/2, 27/2/2, 28/2/1, 29/1/1, 30/1/1, 31/1/1, 32/1/1, 33/1/0, 34/0/0, 35/0/0, 36/0/0}
{
\draw[very thick] (-3.25+.21*\x,-1.55+.4*\y) -- (-3.25+.21*\x+.23,-1.55+.4*\z);
}
\draw (-3.7,-0.75) node {$S_x$};
\draw (-3.35,-1.8) node {\small $j-k+1$};
\draw (3.35,-1.8) node {\small $j+k$};
\draw[<->, thick] (-4, 0) -- (5,0);
\draw (-4.2,0) node {$a$};
\draw[->] (-3.5, 0) -- (-3.5,3.25);
\draw (-3.7,1.625) node {$k$};
\draw (0,-0.1) -- (0,0.1);
\draw (0, -0.3) node {\small $j$};
\filldraw[blue] (-0.13,3.2) rectangle (0.07, 3.4);
\draw[blue, thick] (-0.12,3.21) -- (-3.2,0.13) -- (-3.2, -1.55);
\draw[blue, thick] (0.06,3.21) -- (3.13,0.13) -- (3.13,-1.55) ;
\draw[red, thick] (-3.21, -1.55) -- (3.14,-1.55);
\draw (-0.5,3.7) node {\small $G^k(x)_j$};
%\draw[draw = white, pattern = north east lines] (-3.35,0) -- (-0.1,3.25) -- (0.1,3.25) -- (3.35,0) -- cycle;
\end{tikzpicture}&
\end{tabular}
\end{center}
\caption{Illustration of Lemma \ref{lem:Min}. A strict minimum is reached on $j-k+1$.}
\label{fig:LemMin}
\end{figure}

\begin{proof}
 By induction on $t$, proving only the first equivalence at each step:

\begin{description}
 \item[Base case  $(t=0)$] By definition of $S_x$, $S_x(j+1)<S_x(j) \Leftrightarrow x_j=-1$.

 \item[Induction] Assume that both equivalences hold for a given time $t$. By applying the induction
hypothesis on $G(x)$,
$G^{t+1}(x)_j= -1 \Leftrightarrow j-v_-t+1 = \argmin{[j-v_+t,\ j-v_-t+1]} S_{G(x)}$ and we
conclude by applying Lemma \ref{lem:MarAlea}.
\end{description}
\end{proof}

\subsection{Entry times}\label{sec:entry}

The main result of Section~\ref{sec:particles} implies that, for any $\sigma$-ergodic initial measure $\mu$, $\Lambda_G(\mu)$ contains at most one kind of particle, which one depending on whether $\mu([+1]) > \mu([-1])$ or the opposite. When $\mu([+1]) = \mu([-1])$, $\Lambda_G(\mu)$ only contains the particle-free configuration $\per{0}$. In other words, $G_\ast^t\mu \to \meas{0}$, which means that the probability of seeing a particle in any fixed finite window tends to 0 as $t\to\infty$.

\begin{definition}[Entry times]

 Let $v_-< 0 \leq v_+\in\Z$, $G$ the $(v_-,v_+)$-GA and $x \in\{-1,0,1\}^\Z$. We define: 
 \[T_n^-(x) = \min\{k\in\N\ :\ \exists i\in[0,|v_-|-1],\ G^{k+n}(x)_i=-1\},\] 
 \nomentry{$T_n^+, T_n^-$}{Waiting time after time $n$ before a particle crosses the central column}  with $T_n^-(x) = \infty$ if this set is empty. This is the \define{entry time} in the set $\{b\in\{-1,0,1\}^\Z\ :\ \exists i\in[0,|v_-|-1], b_i=-1\}$ after time $n$ at position 0 starting from $x$. We define $T_n^+(x)$ in a symmetrical manner.
\end{definition}

The size of the considered window is such that any particle ``passing through'' the column 0 appears in this window exactly once (See Figure~\ref{fig:entry}). Of course entry times for particles of speed 0 make no sense. From now on, we only consider $T^-$ for simplicity, all the results being valid for $T^+$.

\begin{figure}[h!]
\begin{center}
 \includegraphics[scale = 0.8]{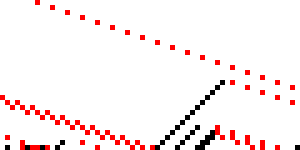}
  \hspace{-9.31cm}
\begin{tikzpicture}
\draw[blue,thick] (0,0) -- (0,4.2);
\draw[blue,thick] (0.42,0) -- (0.42,4.2);
\draw[blue,thick] (0,1) rectangle (0.42,1.14);
\draw[blue,thick] (0,1) -- (0.42,1.14) (0.42,1) -- (0,1.14);
\filldraw[blue!15!white] (0.02,1.16) rectangle (0.40,3.09);
\draw[<->, thick] (-0.2,1.14) -- (-0.2,3.1);
\draw (-0.7, 2.12) node {\small $T_n^-(x)$};
\draw[<->,thick] (-0.2,1) -- (-0.2,0);
\draw (-0.5, 0.5) node {\small $n$};
\draw (-4,-0.01) -- (5,-0.01) (-4,0.13) -- (5,0.13);
\draw (-4.25,0.2) node {\small $x$};
\end{tikzpicture}
\end{center}
\caption{An entry time for the (-3,1)-gliders automaton.}\label{fig:entry}
\end{figure}

As a consequence of Birkhoff's ergodic theorem, when $\mu([-1]) > \mu([+1])$, $-1$ particles persist $\mu$-almost surely and their density converges to a positive number. Therefore:
\begin{itemize}
\itemsep0em
 \item $\mu(T^+_n(x)=\infty) \underset{n\to\infty}{\longrightarrow} 1$;
 \item $\forall \alpha>0,\mu\left(\frac {T_n^-(x)}n \leq \alpha\right)\underset{n\to\infty}{\longrightarrow} 1$,
\end{itemize}

and symmetrically. This is why we only consider the case $\mu([-1]) = \mu([+1])$. K{\r u}rka and al. proved the following result:
\begin{theorem}[\cite{EntryTimes}]
 For the $(-1,0)$-GA (``Asymmetric gliders'') with an initial measure $\mu = \Ber\left(\frac
12,0,\frac 12\right)$:\[\forall \alpha>0,~\mu\left(\frac {T_n^-(x)}n \leq
\alpha\right)\underset{n\to\infty}{\longrightarrow}\frac 2\pi \arctan\sqrt \alpha.\]
\end{theorem}
In the same article, they conjectured that this result could be extended to any initial Bernoulli measure of parameters $(p,1-2p,p)$ for $0\leq p\leq \frac12$ by replacing the right-hand term by $\frac 2\pi \arctan\sqrt{2p\alpha}$. We will prove that this conjecture is actually incorrect.\sk

To state our result, we introduce two particular subclasses of $\Msaz$. We introduce $\alpha$-mixing coefficients of a measure $\mu\in\Msaz$:
\[\alpha_\mu(n) = \sup \{|\mu(A\cap B) - \mu(A)\mu(B)|: A\in \mathcal \Ba_{]-\infty,0]}, B\in \mathcal \Ba_{[n,+\infty[}\}.\]
where $\Ba_{[a,b]}$ is the Borel $\sigma$-algebra generated by the random variables $(X_i)_{a\leq i\leq b}$.

Define:
\begin{itemize}
\itemsep0em
 \item $\Ber_=$ the set of Bernoulli measures on $\{-1, 0, +1\}^\Z$ and parameters $(p,1-2p,p)$ for some
$0<p\leq \frac 12$;
 \item $\Mix$ the set of measures $\mu\in\Ms(\{-1, 0, +1\}^\Z)$ satisfying:
    \begin{itemize}
    \itemsep0em
    \item $\int_{\az} x_0\ \ud\mu(x)  = 0$, i.e., $\mu([-1]) = \mu([+1])$;
    \item $\sum_{k=0}^\infty \int_{\az} x_0\cdot x_k\ \ud\mu(x)$ converges absolutely to a real $\s_\mu^2>0$ (asymptotic variance);
    \item $\exists \varepsilon>0,\sum_{n\geq 0}\alpha_\mu (n)^{\frac 14 - \varepsilon} < \infty$.
    \end{itemize}
\end{itemize}
In particular, $\Ber_= \subset \Mix$. 

\begin{theorem}[Quantitative result for entry time]\label{thm:Entry}

 For any $(v_-,v_+)$-GA with $v_-< 0$ and $v_+\geq 0$ and any initial measure $\mu\in\Mix$,
\[\forall \alpha>0,~\mu\left(\frac{T_n^-(x)}n\leq \alpha\right) \underset{n\to\infty}{\longrightarrow} \frac
2\pi\arctan\left(\sqrt{\frac{-v_-\alpha}{v_+-v_-+v_+\alpha}}\right).\]
\end{theorem}
Notice that this limit is independent from $\mu$ (as long as $\mu\in\Mix$), disproving the conjecture when
$\mu\in\Ber_=$.

\subsection{Brownian motion and proof of the main result}\label{sec:brownian}

The third hypothesis for $\Mix$ is chosen so that the large-scale behaviour of the partial sums $S_x(t)$ can be approximated by a Brownian motion. This invariance principle is the core of our proofs.
The first and second conditions ensure that the Brownian motion obtained this way have no bias and nonzero variance, respectively.

\begin{definition}[Brownian motion]

 A \define{Brownian motion} (or \define{Wiener process}) $B$ of mean $0$ and variance $\sigma^2$ is a continuous
time stochastic process taking values in $\R$ such that:
\begin{itemize}
\renewcommand{\labelitemi}{\labelitemii}
\itemsep0em
 \item $B(0) = 0$,
 \item $t\mapsto B(t)$ is almost surely continuous,
 \item $B(t_2) - B(t_1)$ follow the normal law of mean 0 and variance $(t_2 - t_1)\s^2$;
 \item For $t_1<t_2\leq t'_1<t'_2$, increments $B(t_2) - B(t_1)$ and $B(t'_2) - B(t'_1)$ are independent.
\end{itemize}
\end{definition}
\nomentry{$B(t)$}{Brownian motion value at time $t$}
See \cite{Morters-Peres} for a general introduction to Brownian motion.

\begin{proposition}[Rescaling property]
 
 Let $B$ be a Brownian motion. Then, for any $k>0$, $t\mapsto \frac 1{\sqrt k}B(kt)$ is a Brownian motion with same mean and variance.
\end{proposition}

We now state some invariance principles, which consists in approximating rescaled random walks by Brownian motion.
We use a strong version, which guarantees an almost sure convergence by considering a copy of the process in a richer probability space.

\begin{theorem}[\cite{Shao-Lu2}, Corollary 9.3.1]\label{thm:Invariance}

Let $X = (X_i)_{i\in\N}$ be a family of random variables taking values in $\{-1,0,1\}$. We denote by $\alpha_X(n)$ its $\alpha$-mixing coefficients defined as:

\[\alpha_X(n) = \sup \{|P(A\cap B) - P(A)P(B)|: t\in\N, A\in \Ba_{[0,t]}, B\in \Ba_{[t+n,+\infty[}\},\]
where again $\Ba_{[a,b]}$ is the sigma-algebra generated by $(X_a,\dots, X_b)$.\sk

Assume that:
\begin{enumerate}
 \item $\forall i, \mathbb E(X_i)=0$;
 \item $\frac 1t\mathbb E\left(\sum_{i,j=1}^{\lfloor t\rfloor} X_i\cdot X_j\right)$ converges absolutely to some positive real $\s^2$;
 \item $\exists \varepsilon>0, \sum_{n=1}^\infty \alpha_X(n)^{\frac 14+\varepsilon}$.
\end{enumerate}

Then we can define two processes $X' = (X'_i)_{i\in\N}$ and $B$ on a richer probability space $(\Omega, \mathbb P)$ such that:
\begin{enumerate}
 \item $X$ and $X'$ have the same distribution;
 \item $B$ is a Brownian motion of mean 0, variance $\s^2$;
 \item for any $\varepsilon>0$,
 \[\left|\sum_{i=1}^{\lfloor t\rfloor} X_i - B(t)\right| = O\left(t^{\frac 14 + \varepsilon}\right)\quad\mathbb P\mbox{-almost surely}.\]
\end{enumerate}
\end{theorem}

\begin{corollary}\label{thm:Brown}
Let $\mu \in \Mix$.
For any fixed constants $q<r\in\mathbb R$, we can define a process $X' = (X'_i)_{i\in\Z}$ 
and a family of processes $(t\mapsto B_n(t))_{n\in\N}$ on a richer probability space $(\Omega, \mathbb P)$ such that:
\begin{enumerate}
 \item $X'$ has distribution $\mu$;
 \item every $B_n$ is a Brownian motion of mean 0 and variance $\s_\mu^2>0$;
 \item for any $\varepsilon>0$, denoting by $S_{X'}$ the piecewise linear function defined by $S_{X'}(0) = 0$ and $S_{X'}(k+1)-S_{X'}(k) = X'_k$ for all $k\in\Z$,
 \[\forall n\in\N, \sup_{t\in[q,r]}\left|\frac{S_{X'}(nt)}{\sqrt{n}} - B_n(t)\right| = O\left(n^{-\frac 14 + \varepsilon}\right)\quad\mathbb P\mbox{-almost surely}.\]
\end{enumerate}
\end{corollary}

\begin{proof}
We apply Theorem~\ref{thm:Invariance} on $(X_i)_{i\in\N}$, where $(X_i)_{i\in\Z}$ is distributed according to $\mu$. Because $\mu$ is $\s$-invariant,
this is a stationary process.
The first and third conditions are satisfied by definition of $\Mix$. For the second condition,
\begin{align*}\frac 1n\sum_{0\leq i,j\leq n}|\mathbb E(X_i\cdot X_j)| &= \frac 1n\sum_{i=0}^n\sum_{j=-i}^{n-i}\mathbb E(|X_0\cdot X_{j-i}|)\qquad\textrm{by stationarity}\\
&=\sum_{i=-n}^n\frac{n+1-|i|}n\mathbb E(|X_0\cdot X_{i}|)\qquad\textrm{by reordering the sum}\\
&\sim_{n\to\infty} 2\sum_{i=0}^n\mathbb E(|X_0\cdot X_{i}|) \underset{n\to\infty}\longrightarrow 2\sigma_\mu^2,
\end{align*}
by stationarity and the equivalence criterion for positive series. We obtain two processes $X^1 = (X^1_i)_{i\in\N}$ and $B^1$ on a richer probability space $(\Omega, \mathbb P)$ such that $X^1$ has the same distribution as $x$, $B^1$ is a Brownian motion of mean 0, variance $\s_\mu^2$, and:
 \[\forall \varepsilon>0, \left|\sum_{i=1}^{\lfloor t\rfloor} X^1_i - B^1(t)\right| = \underset{t\to+\infty}O\left(t^{\frac 14 + \varepsilon}\right)\quad\mathbb P\mbox{-almost surely}.\]
Since the variables $X^1_i$ take value in $\{-1,0,1\}$, we have for any $t$ $\left|\sum_{i=1}^{\lfloor t\rfloor} X^1_i - S_{X^1}(t)\right|<1$ 
(a staircase and piecewise linear function having the same values on $\N$). Therefore:
\[\forall \varepsilon>0, \left|S_{X^1}(t) - B^1(t)\right| = \underset{t\to+\infty}O\left(t^{\frac 14 + \varepsilon}\right)\quad\mathbb P\mbox{-almost surely}.\]
\[\forall \varepsilon>0, \forall n\in\N,\ \frac 1{\sqrt n}\left|S_{X^1}(tn)-B^1(tn)\right|=
\underset{n\to\infty}{O}\left(n^{-\frac 14+\epsilon}\right)\cdot\underset{t\to\infty}{O}\left(|t|^{\frac 14+\epsilon}\right) \quad\mathbb P\mbox{-almost surely}. \]
For any $r\in\R_+^2$, taking the sup for $t\in[0,r]$, we obtain:
\[\forall \varepsilon>0, \forall n\in\N,\ \sup_{t\in[0,r]}\left|\frac {S_{X^1}(tn)}{\sqrt n}-\frac {B^1(tn)}{\sqrt n}\right|=\underset{n\to\infty}{O}\left(n^{-\frac14+\epsilon}
\right)\quad\mathbb P\mbox{-almost surely}.\]
By rescaling property $B^1_n : t \mapsto \frac {B^1(tn)}{\sqrt n}$ is a Brownian motion of same mean and variance as $B^1$.\sk

To extend the result to negative values, we apply the theorem again to $(x_{-i-1})_{i\in\N}$, obtaining a process $X^2$ and a Brownian motion $B^2$
satisfying the same asymptotic bound on $t\to -\infty$. Joining both parts, we can see that the process $X' = \dots X^2_{-2}, X^2_{-1}, X^1_{0}, X^1_{1}\dots$ have distribution $\mu$
and $B_n : t\mapsto B_n^1(t)$ if $t\geq 0$, $B_n^2(t)$ if $t<0$ is a Brownian motion. 
\end{proof}

For a survey of invariance principles under different assumptions, see \cite{Merlevede-Rio}.\sk

We now prove the main theorem.

\begin{proof}[Proof of Theorem~\ref{thm:Entry}]

For any $x\in \{-1,0,1\}^\Z$, Lemma \ref{lem:Min} applied on the column 0 gives:
\begin{align*}
T_n^-(x)& =\min\left\{k\in\N\ |\ \exists j\in [0,-v_-[~, S_x(-v_-(n+k)+j+1) <\min_{[-v_+(n+k)+j,\ -v_-(n+k)+j]}S_x\right\}\\
& =\min\left\{k\in\N\ |\ \exists j\in [0,-v_-[~, S_x(-v_-(n+k)+j+1) < \min_{[-v_+(n+k)+j,\ -v_-n]}S_x\right\}
\end{align*}
Indeed, if $(k,j)$ is the smallest pair (in lexicographic order) such that $S_x(-v_-(n+k)+j+1) < \min_{[-v_+(n+k)+j,\ -v_-n]}S_x$, then all pairs $(k',j') < (k,j)$ verify $S_x(-v_-(n+k')+j'+1) \geq \min_{[-v_+(n+k)+j,\ -v_-n]}S_x$, and therefore $S_x(-v_-(n+k)+j+1) < \min_{[-v_+(n+k)+j,\ -v_-(n+k)+j]}S_x$.

Note that if this last condition is reached on $k\in\N$, since $S_x$ is piecewise linear, it is attained for $t$ as soon as $t>k-1$ and reciprocally. Thus:
\[T_n^-(x)=\inf\left\{t\geq 0\ |\ \exists j\in [0,-v_-[~, S_x(-v_-(n+t)+j+2) < \min_{[-v_+(n+t)+j+1,-v_-n]}S_x\right\}\]

Replacing $j$ by $0$ in this expression adds to the infimum a value comprised between $0$ and $\frac {-v_--1}{-v_-}$ (remember $v_-<0$). 
Since the infimum is necessarily an integer, we compensate by taking the integer part:
\begin{align*}
T_n^-(x)&=\left\lfloor\inf\left\{t\geq 0\ |\ S_x(-v_-(n+t)+2) < \min_{[-v_+(n+t)+1,-v_-n]}S_x\right\}\right\rfloor\\
&=\left\lfloor\inf\left\{t\geq 0\ |\ S_x^n\left(-v_-\left(1+\frac tn\right)+\frac2n\right) < \min_{[-v_+(1+\frac
tn)+\frac1n,-v_-]}S_x^n\right\}\right\rfloor\\
&=\left\lfloor n\cdot\inf\left\{t\geq 0\ |\ S_x^n\left(-v_-(1+t)+\frac 2n\right) <
\min_{[-v_+(1+t)+\frac1n,-v_-]}S_x^n\right\}\right\rfloor
\end{align*}
where $S_x^n$ is the rescaled process $t\mapsto\frac{S_x(nt)}{\sqrt n}$. Since $S_x$ is $1$-Lipschitz, $S_x^n$ is $\sqrt n$-Lipschitz (i.e. $|S_k^n(t')-S_k^n(t)|\leq \sqrt n|t'-t|$). Dividing the previous expression by $n$, using the fact that $t-\frac 1n \leq \frac{\lfloor nt\rfloor}n \leq t$ for all $t,n\in\R\times\N$:
\begin{align*}
\mu\left(\min_{[-v_-,-v_-(1+\alpha)]}S_x^n +\frac4{\sqrt n}< \min_{[-v_+(1+\alpha),-v_-]}S_x^n\right) \leq
\mu\left(\frac{T_n^-(x)}n \leq \alpha\right) \tag{1a}\label{eq1a}\\ \mu\left(\frac{T_n^-(x)}n \leq \alpha\right) \leq
\mu\left(\min_{[-v_-,-v_-(1+\alpha)]}S_x^n -\frac3{\sqrt n}< \min_{[-v_+(1+\alpha),-v_-]}S_x^n\right)\tag{1b}\label{eq1b}
\end{align*}

We now bound the left-hand term of (\ref{eq1a}) from below and the right-hand term of (\ref{eq1b}) from above. Using Corollary~\ref{thm:Brown}, we build a process $X'$ and a family of processes $(B_n)_{n\in\N}$ on a richer probability space $(\Omega, \mathbb P)$ such that $X'$ is distributed according to $\mu$ and the $B_n$ are Brownian motions. 
 \[\forall n\in\N, \sup_{[-v_+(1+\alpha),-v_-(1+\alpha)]}\left|\frac{S_{X'}(nt)}{\sqrt{n}} - B_n(t)\right| = 
 O\left(n^{-\frac 14 + \varepsilon}\right)\quad\mathbb P\mbox{-almost surely}.\]
By symmetry, $B_n^l(t) = B_n(-v_--t)-B_n(-v_-)$ and $B_n^r(t) =
B_n(-v_-+t)-B_n(-v_-)$ are two independent Brownian motions on $[0,\ v_--v_+(1+\alpha)]$ and $[0,\ -v_-\alpha]$, respectively. Consequently, for any
$\varepsilon >0$ and $n$ large enough:
\begin{align*}\mu\left(\min_{[-v_-,-v_-(1+\alpha)]}S_x^n +\varepsilon< \min_{[-v_+(1+\alpha),-v_-]}S_x^n\right)&= 
\mathbb P\left(\min_{[-v_-,-v_-(1+\alpha)]}S_{X'}^n +\varepsilon< \min_{[-v_+(1+\alpha),-v_-]}S_{X'}^n\right)\\
&\geq \mathbb P\left(\min_{[-v_-,-v_-(1+\alpha)]}B_n+2\varepsilon<\min_{[-v_+(1+\alpha),-v_-]}B_n\right)\\
&\geq \mathbb P\left(\min_{[0,-v_-\alpha]}B_n^l+2\varepsilon<\min_{[0,-v_-+v_+(1+\alpha)]}B_n^r\right)\tag{2a}\label{eq2}
\end{align*}

and a symmetrical upper bound for (\ref{eq1b}):
\begin{align*}\mu\left(\min_{[-v_-,-v_-(1+\alpha)]}S_x^n -\varepsilon< \min_{[-v_+(1+\alpha),-v_-]}S_x^n\right)
&\leq \mathbb P\left(\min_{[0,-v_-\alpha]}B_n^l-2\varepsilon<\min_{[0,-v_-+v_+(1+\alpha)]}B_n^r\right)\tag{2b}\label{eq2b}
\end{align*}

We now evaluate the right-hand expression in (\ref{eq2}).\sk

For any Brownian motion $B$ and $b>0$, we have by rescaling $\displaystyle\mathbb P\left(\min_{[0,b]} B\geq m\right) = 
\mathbb P\left(\min_{[0,1]} B\geq \frac m{\sqrt b}\right)$. Furthermore,
since $B_n^l$ and $B_n^r$ are independent, so are $\displaystyle\min_{[0,1]} B^l_x$ and $\displaystyle\min_{[0,1]}
B^r_x$. Denote by $\mu_m$ the law of the minimum of a Brownian motion on $[0,1]$, which is defined by the density function:
\[\begin{array}{rcll}
  \R&\to&\R&\\
  t&\mapsto &\frac{e^{-t^2}}2&\mbox{ if }t\leq 0,\\
  &&0 &\mbox{ otherwise.}
  \end{array}\quad\mbox{(see \cite{Morters-Peres}).}
\]
This means that for any $y,z>0$:

\begin{align*}
\mathbb P\left(\min_{[0,y]}B^l_n< \min_{[0,z]} B^r_n\right) &=\int_{-\infty}^0\int_{-\infty}^01_{\{\sqrt y\cdot m_1\leq \sqrt
z\cdot m_2\}}d\mu_m(m_2)d\mu_m(m_1)\\
&\underset{(i)}=\frac 4{2\pi}\int_{-\infty}^0\int_{-\infty}^{\frac {\sqrt{z}\cdot m_2}{\sqrt y}}
e^{\frac{-m_1^2}2}e^{\frac {-m_2^2}2} dm_1dm_2\\
&\underset{(ii)}=\frac 2{\pi}\int_\pi^{\pi+\arctan\left(\sqrt {\frac yz}\right)}\int_0^{+\infty} re^{\frac{-r^2}2}
drd\theta\\
&=\frac 2{\pi}\arctan\left(\sqrt {\frac yz}\right)\label{eq3}\tag{3}
\end{align*}

(i) by using the law of the minimum of a Brownian motion, (ii) by passing in polar variables. For $\varepsilon>0$, a
similar calculation gives:

\begin{align*}\mathbb P\left(\min_{[0,y]}B^l_n<\min_{[0,z]}B^r_n\right) - \mathbb P\left(\min_{[0,y]}B^l_n+2\varepsilon< \min_{[0,z]}B^r_n\right) 
&\leq \frac 4{2\pi}\int_{-\infty}^0\int_{\frac {\sqrt z\cdot m_2-2\varepsilon}{\sqrt y}}^{\frac
{\sqrt z\cdot m_2}{\sqrt y}} e^{\frac{-m_1^2}2}e^{\frac {-m_2^2}2} dm_1dm_2\\
&\leq\frac {8\varepsilon}{2\pi\sqrt y}\int_{-\infty}^0 e^{\frac{-ym_2^2}{2z}}e^{\frac {-m_2^2}2} dm_2\\
&\underset{\varepsilon \to 0}{\longrightarrow}0\label{eq4}\tag{4}
\end{align*}

Using (\ref{eq3}) and (\ref{eq4}), we see that the right-hand term in (\ref{eq2}) converges to $\frac 2\pi\arctan\left(\sqrt{\frac{-v_-\alpha}{v_+-v_-+v_+\alpha}}\right)$ as $\varepsilon\to0$. The left-hand term in (\ref{eq2b}) can be bounded from above by the same method. We apply this result to (\ref{eq1a}) and (\ref{eq1b}) by taking $\varepsilon = \frac {3}{\sqrt n}$ (resp. $\frac {4}{\sqrt n}$), and the theorem follows.
\end{proof}

\subsection{Particle density}\label{sec:density}
\begin{definition}[Particle density in a configuration]

 The \define{$-1$ particle density} in $x\in \{-1,0,1\}^\Z$ is defined as $d_-(x) = \freq(-1,x)$.
$d_+(x)$ is defined in a symmetrical manner.
\end{definition}
In all the following, any result on $d_-$ also holds for $d_+$ by symmetry.

\begin{theorem}[Decrease rate of the particle density]\label{thm:density}

Let $G$ be a $(v_-,v_+)$-GA with initial measure $\mu\in\Mix$. Then:
 \[\textrm{For $\mu$-almost all } x\in\{-1,0,1\}^\Z,\ \forall \varepsilon > 0,\ d_-(G^t(x)) = O \left(t^{-\frac 14+\varepsilon}\right)\]

If furthermore $\mu\in\Ber_=$:
 \[\textrm{For $\mu$-almost all } x\in\{-1,0,1\}^\Z,\ d_-(G^t(x)) \sim t^{-\frac 12}\]
\end{theorem}

\begin{proof}
 When $\mu\in \Mix$, it is in particular $\s$-ergodic, and so are its images $\G^t\mu$. By Birkhoff's ergodic theorem, one has $d_-(G^t(x)) = \G^t\mu([-1]) = \mu(G^t(x)_0=-1)$ for $\mu$-almost all $x\in\{-1,0,1\}^\Z$.\sk

We first prove the theorem when $G$ is the $(-1,0)$-gliders automaton. By Lemma \ref{lem:Min},
\[\mu(G^t(x)_0= -1)=\mu\left(S_x(t+1) < \min_{[0,t]} S_x\right).\]
\paragraph{Equivalence ($\mu\in\Ber_=$):}~
By symmetry, \[\mu\left(S_x(t+1) < \min_{[0,t]} S_x\right) = \mu\left(S_x(0) < \min_{[1,t+1]} S_x\right),\] which is the probability that the random walk starting from 0 remains strictly positive during $t$ steps, also known as its probability of survival. According to \cite{Redner}, when the random walk is symmetric and the steps are independent, we have the equivalence
$\mu(G^t(x)_0=-1) \sim \frac1{\sqrt t}$. 
\paragraph{Upper bound:}
\[\mu\left(S_x(t+1) < \min_{[0,t]} S_x\right) \leq \mu\left(S_x^{t+1}(1) = \min_{[0,1]} S^{t+1}_x\right).\]

Using Corollary~\ref{thm:Brown}, we have:
\begin{align*}
\mu\left(S_x^{t+1}(1)=\min_{[0,1]} S^{t+1}_x\right) &= \mathbb P\left(S_{X'}^{t+1}(1)=\min_{[0,1]} S^{t+1}_{X'}\right)\\
&\leq \mathbb P\left(B_{t+1}(1)\leq\min_{[0,1]}B_{t+1}+C_{t+1}\right)\\
&\leq\mathbb P\left(B_{t+1}(0)\leq\min_{[0,1]}B_{t+1}+C_{t+1}\right),
\end{align*}
where $\displaystyle C_{t+1} = \sup_{[0,1]}\left|S_{X'}^{t+1}-B_{t+1}\right| = O\left(t^{-\frac 14+\varepsilon}\right)$ $\mathbb P$-almost surely, 
and where the third line is obtained by symmetry of the Brownian motion.

Furthermore $\displaystyle\mathbb P\left(\min_{[0,1]}B_{t+1}>-C_{t+1}\right) = \int_{-C_{t+1}}^0 e^{-x^2/2} dx \leq C_{t+1} = O \left(t^{-\frac
14+\varepsilon}\right)$.

\paragraph{General case (any $v_- < v_+$):} Let $G'$ be the $(v_-,v_+)$-GA. Then
\[G' = \s^{v_+}\circ G^{v_+-v_-}.\]
To conclude, it is enough to see that the particle density is $\s$-invariant and decreasing under the action of $G$.
\end{proof}

\subsection{Rate of convergence}\label{sec:rate}

In this section, we estimate the rate of convergence to the limit measure. For that we fix a distance on the space $\Msaz$ of $\s$-invariant measures, which induces the weak$^\ast$ topology:
\[\dm(\mu,\nu)=\sum_{n\in\N}\frac{1}{2^n}\max_{u\in\A^n}\left|\mu([u])-\nu([u])\right|.\]

\begin{theorem}[Rate of convergence to the limit measure]\label{thm:rate}

Let $G$ be the $(v_-,v_+)$-GA with initial measure $\mu\in\Mix$. Then:
\[ \forall \varepsilon>0,\ \dm(G^t_\ast\mu,\meas{0})=O\left(t^{-1/4+\varepsilon}\right)\] 

If furthermore $\mu\in\Ber_=$: \[\dm(G^t_\ast\mu,\meas{0})=\Omega\left(t^{-1/2}\right)\]
\end{theorem}

\begin{proof}
We first prove the theorem when $G$ is the $(-1,0)$-gliders automaton. By defining $0^\ell\in\A^\ell$ the word
containing only zeroes, the distance can be rewritten:
\[\forall t\in\N, \dm(G_\ast^t\mu,\meas{0}) = \sum_{\ell=1}^\infty\frac
1{2^\ell}G^t_\ast\mu\left(\az\backslash[0^\ell]\right).\]

\textbf{Lower bound when $\mu\in\Ber_=$:} $\dm(G_\ast^t\mu,\meas{0}) >
G^t_\ast\mu\left(\az\backslash[0]\right)$. We conclude with Theorem \ref{thm:density}.\

\textbf{Upper bound:} We give an upper bound for $G^t_\ast\mu(\az\backslash[0^\ell]) = \mu(\exists 0\leq d\leq \ell, G^t(x)_d=\pm1)$ for $\ell\in\N$
and $t\in\N$. By Lemma \ref{lem:Min},
\[\forall d\in\Z,\ G^t(x)_d = +1 \Leftrightarrow  S_x(d)<\min_{[d+1,d+t]}S_x.\]
Therefore:
\begin{align*}
  G^t_\ast\mu\left(\bigcap_{d=0}^\ell [+1]_d\right)&\leq \mu\left(\min_{[0,\ell]} S_x < \min_{[\ell+1,t]}S_x\right)\\
  &\leq \mu\left(\min_{[0,t]}S_x \geq -\ell\right)\\
  &\leq \mu\left(\min_{[0,1]}S^t_x \geq -\frac \ell{\sqrt t}\right)
\end{align*}  
By Corollary~\ref{thm:Brown}, using the same notations as in the previous proofs:  
\begin{align*}
  G^t_\ast\mu\left(\exists 0\leq d\leq \ell, x_d = +1\right) &\leq \mathbb P\left(\min_{[0,1]}S^t_{X'} \geq -\frac \ell{\sqrt t}\right)\\
  &\leq \mathbb P\left(\min_{[0,1]}B_t \geq -\frac \ell{\sqrt t} - C_t\right) \quad\quad\textrm{ where }C_t = O\left(t^{-\frac
14+\varepsilon}\right)\\
 &= O\left(t^{-\frac 14+\varepsilon}\right)
\end{align*}
for any $\varepsilon>0$, following the same calculations as in Section~\ref{sec:density}. The case of $-1$ particles is symmetrical, and we conclude.

\paragraph{General case:} Apply the same method as in the previous section, considering that $\dm$ and all
considered measures are $\s$-invariant and that any CA is Lipschitz w.r.t $\dm$.
\end{proof}

\subsection{Extension to other cellular automata}\label{sec:extensions}

\begin{definition}Let $F_1, F_2$ be two CAs on $\az$ and $\B^{\Z}$, respectively. We say that $F_1$
\define{factorises onto} $F_2$ if there exists a \define{factor} $\pi: \az\to \mathcal \B^{\Z}$ such that
$\pi\circ F_1 = F_2 \circ\pi$.\end{definition}

In particular, if $F_2$ admits a particle system $(\P, \pi_2, \phi)$, then $F_1$ admits a particle system with $(\P, \pi\circ\pi_2, \phi)$.

In this section, we extend the Theorems~\ref{thm:Entry} and \ref{thm:density} to automata that factorise onto a gliders automaton, and discuss conditions for the extension of Theorem~\ref{thm:rate}. In Section \ref{sec:defects}, we exhibited a general method to find such a factor using experimental intuition when such a factor is not obvious. In other words, we extend our results to automata that admit a particle system $(\P,\pi,\phi)$, where $\P = \{-1,+1\}$ and $\phi$ updates the particle positions similarly to a gliders automaton.\sk

In order to extend the theorem to such CAs, starting from an initial measure $\mu$, we must first ensure that $\pi_\ast\mu\in\Mix$. We show that the third condition in the definition of $\Mix$ is invariant under morphism.

\begin{proposition}
 Let $\pi: \az \to \bz$ be a morphism, $\mu\in\Msaz$ and $k>0$ any real such that $\sum_{n\geq 0}\alpha_\mu (n)^k <
\infty$. Then, $\sum_{n\geq 0}\alpha_{\pi_\ast\mu} (n)^k < \infty$.
\end{proposition}

\begin{proof}
We keep the notations from the definition of $\alpha_\mu(n)$. $\pi$ is defined by a local rule with neighbourhood $\Nb \subset [-r,r]$ for some $r>0$. Then, $\pi^{-1}\Ba_{]-\infty,0]} \subset \Ba_{]-\infty,r]}$ and $\pi^{-1}\Ba_{[n, +\infty[} \subset \Ba_{[n-r, +\infty[}$. By $\s$-invariance, we have for all $n$ $\alpha_{\pi_\ast\mu}(n) < \alpha_\mu(n-2r)$, and the result follows.
\end{proof}

Hence, if $\mu\in\Mix$, we only have to prove that $\pi_\ast\mu$ weighs evenly the sets of particles $-1$ and $+1$, and that the corresponding asymptotic variance is not zero. Under these assumptions, we can extend some of the previous results with the forbidden patterns playing the role of the particles.

\begin{corollary}
 Let $F:\az\to\az$ be a CA and $\mu\in\Msaz$. Suppose that $F$ factorises onto a $(v_-,v_+)$-GA via a factor $\pi$ such that $\pi_\ast\mu\in\Mix$.

Then Theorem~\ref{thm:Entry} and the first point of Theorem~\ref{thm:density} hold if we replace ``$x_k = \pm1$'' by ``$\pi(x)_k=\pm1$''.
\end{corollary}

Even if $\mu$ is a simple, e.g. Bernoulli measure, $\pi_\ast\mu$ can fail to satisfy the first and second condition of $\Mix$. We provide a counterexample at the end of this section.

\paragraph{Examples:}

\begin{figure}[!ht]
 \begin{center}
\includegraphics[width = 0.47\textwidth, trim = 0 0 0 50, clip]{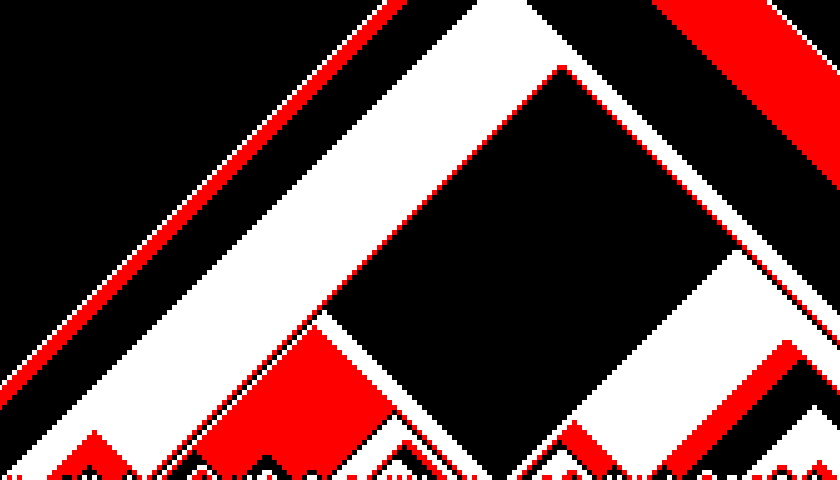}
\hspace{0.5cm}
\includegraphics[width = 0.47\textwidth, trim = 0 0 0 50, clip]{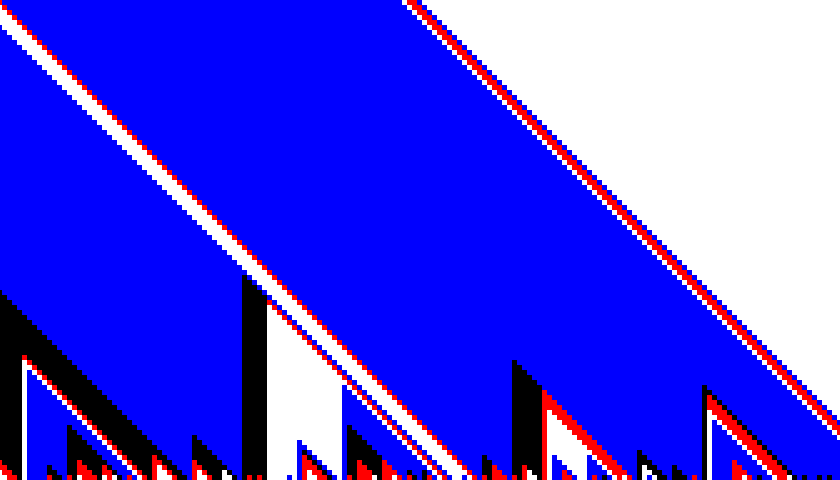}
\hspace{0.5cm}
\includegraphics[width = 0.47\textwidth, trim = 0 0 0 50, clip]{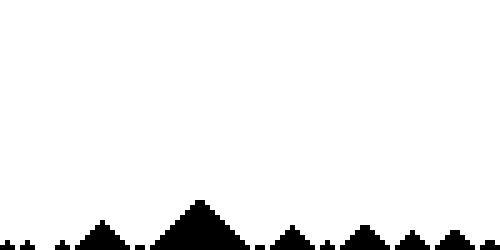}
 \end{center}
\caption{The 3-state cyclic CA, a one-sided captive CA and the product CA.}
\end{figure}

\begin{description}
\item[Traffic automaton] Let $\A = \{0,1\}$ and $F_{184}$ be the elementary CA corresponding to rule $\#184$ as defined in Section~\ref{section.184}.
$F_{184}$ factorises on the $(-1,+1)$-gliders automaton, using the factor introduced in that section:
\[\begin{array}{ccc}
   00 &\mapsto& +1\\
   11 &\mapsto& -1\\
   01,10 &\mapsto& 0
  \end{array}
\]

This factor is represented in Figure~\ref{fig:Factors}. If $\mu$ is a measure such
that $\pi_\ast\mu\in\Mix$, then Theorem~\ref{thm:Entry} and the first point of Theorem~\ref{thm:density} hold.

For example, this is true for the 2-step Markov measure defined by the
matrix $\left(\begin{matrix}p&1-p\\1-p&p\end{matrix}\right)$ and the eigenvector $\left(\begin{matrix}1/2\\
1/2\end{matrix}\right)$ with $p>0$. A particular case is the Bernoulli measure of parameters $(\frac 12,\frac 12)$.

The upper bound in Theorem~\ref{thm:rate} can also be extended by considering the fact that 

\[\dm(F_{184\ast}^t\mu, \meas{01}) \leq 2\sum_{n\in\N}\frac{1}{2^n}\pi_\ast F_{184\ast}^t\mu\left(\A^\Z \backslash [0^n]\right).\]
From there the upper bound can be obtained as in the original proof.

\item[3-state cyclic automaton] Let $\A = \Z/3\Z$ and $C_3$ be the 3-state cyclic automaton. We
consider the factor $\pi$ defined in Section~\ref{sec:particlesexamples}:
\[\begin{array}{cccl}
   ab &\mapsto& +1 &\mbox{ if } a=b+1\mod 3\\
   ab &\mapsto& -1 &\mbox{ if } a=b-1\mod 3\\
   ab &\mapsto& 0 &\mbox{ if } a=b
  \end{array}
\]
If $\mu$ is such that $\pi_\ast\mu\in\Mix$, then Theorem~\ref{thm:density} applies. This is true in particular when $\mu$ is any 2-step Markov measure defined by a matrix $(p_{ij})_{1\leq i,j\leq 3}$ satisfying $p_{01}+p_{12}+p_{20} = p_{10}+p_{21}+p_{02}$, all of these values being nonzero, with $(\mu_i)_{1\leq i\leq 3}$ its only eigenvector. This includes any nondegenerate Bernoulli measure. However, even when the limit measure is known (e.g. starting from the uniform measure),
Theorem~\ref{thm:rate} does not apply directly.

\item[One-sided captive automata] Let $F$ be any one-sided captive cellular automaton defined by a local rule $f$. 
As explained in Section \ref{sec:particlesexamples}, $F$ factorises onto the $(-1,0)$-gliders automaton with a factor defined by:
\[\begin{array}{cccl}
   ab &\mapsto& +1 &\mbox{ if } a\neq b, f(a,b)=a\\
   ab &\mapsto& -1 &\mbox{ if } a\neq b, f(a,b)=b\\
   ab &\mapsto& 0 &\mbox{ if } a=b
  \end{array}
\]
For an initial measure $\mu$, if
$\pi_\ast\mu\in\Mix$, then Theorem~\ref{thm:Entry} and the first point of Theorem~\ref{thm:density} apply.\sk

Notice that this class of automata contains the identity ($\forall a,b\in\A, f(a,b) = b$) and the shift $\s$ ($\forall
a,b\in\A, f(a,b) = a$). However, since we have in each case $\pi^{-1}(+1) = \emptyset$ or $\pi^{-1}(-1) = \emptyset$, it
is impossible to find an initial measure that weighs evenly each kind of particle, and so $\pi_\ast\mu$ cannot belong in
$\Mix$.
The limit measure, however, depends on the exact rule, and Theorem~\ref{thm:rate} does not apply directly.
\end{description}

\paragraph{Counter-example:}
\begin{description}
\item [Product automaton] Let $\A = \Z/2\Z$ and $F_{128}$ be the CA of neighbourhood $\{-1,0,1\}$ defined by the local
rule
$f(x_{-1},x_0,x_{1}) = x_{-1}\cdot x_0\cdot x_{1}$. Using the formalism from Section \ref{sec:defects}, we can see that
$F_{128}$ factorises onto the $(-1,1)$-GA by the factor
\[\pi:\left\{\begin{array}{ccc} 01&\to&+1\\
				10&\to&-1\\
				\textrm{otherwise}&\to&0
             \end{array}\right.
\]
If $\mu$ is any Bernoulli measure, then $\pi_\ast\mu$ satisfies all conditions of $\Mix$ except that $\s_\mu = 0$;
indeed, we can check that for $\pi_\ast\mu$-almost all configurations, the particles $+1$ and $-1$ alternate. Hence, only one particle can cross any given column after time 0, and therefore $\forall \alpha>0,\ \mu\left(\frac{T_n^-(x)}n\leq \alpha\right)\underset{n\to\infty}{\longrightarrow} 0$. Furthermore, any particle survives up to time $t$ only if it is the border of a initial cluster of black cells larger than $2t$ cells, which happens with a probability $\mu([1])^{2t}$ decreasing exponentially in $t$.
\end{description}

Even though we showed that the asymptotic distributions of entry times are known for some class of cellular automata and a large class of measures, this covers only very specific dynamics. It is not known how these results extend for more than 2 particles and/or other kinds of particle interaction. In particular, there is no obvious stochastic process characterising the behaviour of such automata that would play the role of $S_x$ in our proofs.

%%%%
%%%%%%%%%%%%%%%%%%%%%%%%%%%%%%%%%%%%%%%%%%%
%%%%%%%%%%%%%%%%%%%%%%%%%%%%%%%%%%%%%%%%%%%
%%%%

\subsection*{Acknowledgements}
This work was partially supported by the ANR project QuasiCool (ANR-12-JS02-011-01) and the ANR project Valet (ANR-13-JS01-0010). The first author acknowledges the financial support of Basal project No. PFB-03 CMM, Universidad de Chile. We also thank two anonymous referees for their careful reading and many remarks.

\newpage
\bibliographystyle{alpha}
\bibliography{Biblio}

\end{document}